\documentclass[10pt,letterpaper]{amsart}

\usepackage[dvipsnames]{xcolor}
\usepackage[a4paper,top=3cm,bottom=2.5cm,left=3cm,right=3cm,marginparwidth=1.75cm]{geometry}

\usepackage{amsmath, amssymb, amsfonts, latexsym, mdwlist, amsthm}
\usepackage{subfig}
\usepackage{graphicx}
\usepackage{tikz-cd}
\usepackage{wrapfig}
\usepackage{mathrsfs}

\tikzstyle{startstop} = [rectangle, rounded corners, minimum width=3cm, minimum height=1cm,text centered, draw=black, fill=red!30]
\tikzstyle{io} = [trapezium, trapezium left angle=70, trapezium right angle=110, minimum width=3cm, minimum height=1cm, text centered, draw=black, fill=blue!30]
\tikzstyle{process} = [rectangle, minimum width=3cm, minimum height=1cm, text centered, draw=black, fill=orange!30]
\tikzstyle{decision} = [diamond, minimum width=3cm, minimum height=1cm, text centered, draw=black, fill=green!30]
\tikzstyle{arrow} = [thick,->,>=stealth]

\usepackage{mathtools}
\usepackage{tikz}
\usepackage{comment}

\usetikzlibrary{patterns}

\usepackage{enumerate}

\usepackage{amsmath, amssymb, amsfonts, latexsym, mdwlist, amsthm}
\usepackage{subfig}
\usepackage{wrapfig}

\usepackage{mathtools}
\usepackage{bm}

\definecolor{lightred}{HTML}{ff4d4d}
\definecolor{lightblue}{HTML}{1F88CD}
\definecolor{lightgrey}{HTML}{727272}
\definecolor{lightblue2}{HTML}{009EC1}
\definecolor{mypink}{HTML}{FD00B0}

\usepackage[bookmarks, colorlinks, breaklinks, pdftitle={},
pdfauthor={Zhiyu Liu}]{hyperref}
\hypersetup{linkcolor=OliveGreen,citecolor=MidnightBlue,filecolor=black,urlcolor=blue}

\usepackage{tikz}
\usetikzlibrary{calc,trees,positioning,arrows,chains,shapes.geometric,%
    decorations.pathreplacing,decorations.pathmorphing,shapes,%
    matrix,shapes.symbols}

\tikzset{
>=stealth',
  punktchain/.style={
    rectangle,
    rounded corners,
    draw=black, thick,
    minimum height=3em,
    text centered,
    on chain},
  line/.style={draw, thick, <-},
  element/.style={
    tape,
    top color=white,
    bottom color=blue!50!black!60!,
    minimum width=8em,
    draw=blue!40!black!90, very thick,
    text width=10em,
    minimum height=3.5em,
    text centered,
    on chain},
  every join/.style={->, thick,shorten >=1pt},
  decoration={brace},
  tuborg/.style={decorate},
  tubnode/.style={midway, right=2pt},
}

\usepackage{paralist}
\setdefaultenum{(a)}{(i)}{}{}
\usepackage[shortlabels]{enumitem} 

\makeatletter
\newtheorem*{rep@theorem}{\rep@title}
\newcommand{\newreptheorem}[2]{%
\newenvironment{rep#1}[1]{%
 \def\rep@title{#2 \ref{##1}}%
 \begin{rep@theorem}}%
 {\end{rep@theorem}}}
\makeatother

\newtheorem{theorem}{Theorem}[section]
\newreptheorem{theorem}{Theorem}
\newtheorem{proposition}[theorem]{Proposition}

\newtheorem{lemma}[theorem]{Lemma}

\newtheorem{corollary}[theorem]{Corollary}
\newreptheorem{corollary}{Corollary}
\newtheorem{conjecture}[theorem]{Conjecture}
\newreptheorem{conjecture}{Conjecture}

\newtheorem{thm-int}{Theorem}

\theoremstyle{definition}
\newtheorem{Def-s}[theorem]{Definition}
\newtheorem{definition}[theorem]{Definition}

\newtheorem{remark}[theorem]{Remark}

\newcommand{\ignore}[1]{}

\usepackage[numbers]{natbib}
\newcommand{\ra}{\rightarrow}

\newcommand{\sst}{\subset}

\newcommand{\D}{\mathrm{D}}

\newcommand{\ZZ}{\mathbb{Z}}

\newcommand{\QQ}{\mathbb{Q}}

\newcommand{\CC}{\mathbb{C}}
\newcommand{\PP}{\mathbb{P}}

\newcommand{\kk}{\mathbf{k}}

\newcommand{\DT}{\mathsf{DT}}
\newcommand{\GW}{\mathsf{GW}}
\newcommand{\PT}{\mathsf{PT}}
\newcommand{\GV}{\mathsf{GV}}

\newcommand{\NS}{\mathrm{NS}}

\newcommand{\ch}{\mathrm{ch}}

\newcommand{\td}{\mathrm{td}}

\newcommand{\KK}{\mathrm{K}}

\newcommand{\gedim}{\mathrm{gedim}}

\newcommand{\gmax}{g_{\mathrm{max}}}

\renewcommand{\Re}{\operatorname{Re}}
\renewcommand{\Im}{\operatorname{Im}}

\DeclareMathOperator{\cok}{cok}

\DeclareMathOperator{\Coh}{\mathrm{Coh}}

\DeclareMathOperator{\sF}{\mathsf{F}}

\DeclareMathOperator{\Pic}{Pic}

\DeclareMathOperator{\Gr}{Gr}

\newcommand{\cC}{\mathcal{C}}
\newcommand{\cA}{\mathcal{A}}

\newcommand{\cH}{\mathcal{H}}

\DeclareMathOperator{\oh}{\mathcal{O}}
\newcommand{\zy}[1]{\textcolor{red}{#1}}

\usepackage{hyperref}
\hypersetup{
	colorlinks=true,
    linkcolor={blue},
    citecolor={lightred},
	urlcolor={black}
}

\begin{document}

\title[Castelnuovo bound for curves in projective 3-folds]{Castelnuovo bound for curves in projective 3-folds}
\subjclass[2020]{14N35 (Primary); 14J33, 14H50, 14F08 (Secondary)}
\keywords{Castelnuovo bound, Gromov--Witten invariants, Gopakumar--Vafa invariants,  Donaldson--Thomas invariants.}

\author{Zhiyu Liu}

\address{School of Mathematical Sciences, Zhejiang University, Hangzhou, Zhejiang Province, 310058 P. R. China}
\email{jasonlzy0617@gmail.com}
\urladdr{sites.google.com/view/zhiyuliu}


\begin{abstract}
The Castelnuovo bound conjecture, which is proposed by physicists, predicts an effective vanishing result for Gopakumar--Vafa invariants of Calabi--Yau 3-folds of Picard number one. Previously, it is only known for a few cases and all the proofs rely on the Bogomolov--Gieseker conjecture of Bayer--Macr\`i--Toda.

In this paper, we prove the Castelnuovo bound conjecture for any Calabi--Yau 3-folds of Picard number one, up to a linear term and finitely many degree, without assuming the conjecture of Bayer--Macr\`i--Toda. Furthermore, we prove an effective vanishing theorem for surface-counting invariants of Calabi--Yau 4-folds of Picard number one. We also apply our techniques to study low-degree curves on some explicit Calabi--Yau 3-folds.

Our approach is based on a general iterative method to obtain upper bounds for the genus of one-dimensional closed subschemes in a fixed 3-fold, which is a combination of classical techniques and the wall-crossing of weak stability conditions on derived categories, and works for any projective 3-fold with at worst isolated singularities over any algebraically closed field.

\end{abstract}

\vspace{-1em}
\maketitle

\setcounter{tocdepth}{1}
\tableofcontents

\section{Introduction}




One of the most important and difficult problems in geometry and physics is determining Gromov--Witten invariants of projective Calabi--Yau 3-folds. The effort to compute genus zero invariants of Calabi--Yau 3-folds led to the birth of mirror symmetry as a mathematical subject \cite{Givental1998,lian-liu-yau}. More than ten years ago, a group of physicists shocked the community by announcing a physical derivation of Gromov--Witten invariants of a series of projective Calabi--Yau 3-folds of Picard number one in \cite{HKQ09}, up to a very high genus.

Their physical derivation used five mathematical conjectures: Four BCOV conjectures from the B-model and the Castelnuovo bound conjecture from the A-model. After applying all BCOV conjectures, we only need to determine finitely many initial conditions to calculate the Gromov--Witten generating series at each genus. Then the Castelnuovo bound conjecture, roughly speaking, fixes a large number of remaining initial conditions.

During the last few years, great progress has been achieved in proving BCOV conjectures. See e.g.~\cite{GJR17P,GJR18,Chen_2021,Chang_2021,CJR22,CGL18,CGLFeynman}. The main content of this paper is to study the A-model conjecture, i.e.~the Castelnuovo bound.


Before stating the conjecture, we shall fix some notations. The degree of a smooth projective variety $X$ of Picard number one is defined as the self-intersection number of the ample generator of its Neron--Severi group $\NS(X)$. When $X$ is a projective Calabi--Yau 3-fold of Picard number one, we denote by $\GV_{g,d}$ the genus $g$ degree $d$ Gopakumar--Vafa (GV) invariant of $X$ (cf.~Section \ref{subsec-vanish}).

The Castelnuovo bound conjecture predicts an effective vanishing result for $\GV_{g,d}$:

\begin{conjecture}[{Castelnuovo bound conjecture}]\label{conj-cast}
Let $X$ be a projective Calabi--Yau 3-fold of Picard number one and degree $n$. Then $\GV_{g,d}=0$ when
\[g > \frac{1}{2n}d^2+\frac{1}{2}d+1.\]
\end{conjecture}

Previously, this conjecture is only known for some special cases. More precisely, Conjecture \ref{conj-cast} is proven in \cite{liu-ruan:cast-bound,soheyla:cast} under the assumption that $X$ satisfies Bayer--Macr\`i--Toda's (BMT) conjecture \cite{bayer2011bridgeland,bayer2016space} (cf.~Conjecture \ref{SBG_conj}). However, the BMT conjecture is only known for a very few examples of Calabi--Yau 3-folds, such as quintic 3-folds \cite{liquintic} and $(2,4)$-complete intersections in $\PP^5$ \cite{liu:2-4}, and proving it for a large class of projective Calabi--Yau 3-folds of Picard number one seems still far beyond the current techniques.



In this paper, we prove the following effective vanishing theorem for GV-invariants of any projective Calabi--Yau 3-fold $X$ of Picard number one, without assuming the BMT conjecture for $X$.

\begin{theorem}[{Theorem \ref{thm-vanishing-GV}}]\label{thm-intro-GV}
Let $X$ be a projective Calabi--Yau 3-fold of Picard number one and degree $n$. Then $\GV_{g,d}=0$ when
\[g> \frac{1}{2n}d^2+\frac{nm^3-4m}{2}d+1\]
and $d\geq N$, where $N$ is an explicit integer and $m$ is any integer such that $mH$ is very ample for the ample generator $H$ of $\NS(X)$.
\end{theorem}

This solves Conjecture \ref{conj-cast} up to a linear term and finitely many degree, and also answers a question of Doan--Ionel--Walpuski \cite[Question 1.6]{DIW21} for these 3-folds. The same vanishing result holds for Donaldson--Thomas invariants and Pandharipande--Thomas invariants as well (cf.~Theorem \ref{thm-vanishing}).


We refer to Theorem \ref{thm-vanishing-GV} for the definition of $N$ in Theorem \ref{thm-intro-GV} and a sharper bound. In particular, if the ample generator of the Neron--Severi group $\NS(X)$ is very ample, then we can take $m=1$. When $X$ is a quintic 3-fold, the bound in Theorem \ref{thm-intro-GV} coincides with the one in Conjecture \ref{conj-cast}.

We also have similar results for surface-counting invariants of Calabi--Yau 4-folds, which will be discussed in Section \ref{intro-subsec-4fold}.

\subsection{Bounds of the arithmetic genus}\label{intro-bound-genus}

Our Theorem \ref{thm-intro-GV} is a consequence of an upper bound of the (arithmetic) genus of $1$-dimensional closed subschemes in a fixed 3-fold. For simplicity, we introduce the following notations. 

We fix the base field $\kk$ to be an algebraically closed field of \emph{any characteristic}. We say $X$ is a \emph{projective 3-fold} if $X$ is a projective pure-dimensional scheme over $\kk$ with $\dim X=3$. For any projective 3-fold $X$ of Picard number one and $1$-dimensional closed subscheme $C\subset X$, the degree of $C$ is defined to be the intersection number $\deg(C):=C.H$, where $H$ is the ample generator of $\NS(X)$. The (arithmetic) genus of $C$ is $g(C):=1-\chi(\oh_C)$.

\begin{definition}
Let $X$ be a projective 3-fold of Picard number one. Then we define
\[\gmax^{X}(d):=\max\{g(C)~|~C\subset X \text{ is a }1\text{-dimensional closed subscheme of } \deg(C)=d\}.\]
\end{definition}


Bounds of the genus of reduced curves are investigated by many authors via classical techniques, see e.g.~\cite{harris:geometry-of-curve-I,harris:curve-in-space}. On the other hand, an optimal upper bound of $\gmax^{X}(d)$ for $X=\PP^3$ is proven in \cite{Hartshorne1994TheGO}. Beyond these results, the BMT conjecture and the wall-crossing of (weak) stability conditions on derived categories are applied in \cite{macri:space-curve,liu-ruan:cast-bound,soheyla:cast} to get an upper bound of $\gmax^{X}(d)$ when $X$ is a smooth projective 3-fold of Picard number one satisfying the BMT conjecture, which can be used to deduce Conjecture \ref{conj-cast} in this case.

In this paper, we combine the classical approach and the wall-crossing method to prove the following upper bound of $\gmax^{X}(d)$ without assuming the smoothness and the BMT conjecture for $X$:



\begin{theorem}[{Corollary \ref{cor-pic-rk-1}}]\label{intro-thm-pic-1}
Let $X$ be a factorial projective 3-fold over $\kk$ with at worst isolated singularities of Picard number one and degree $n$, and $H$ be the ample generator of $\NS(X)$. Let $s$ be the least integer such that $|sH|\neq \varnothing$ and $m$ be any integer with $mH$ very ample. 

Then there exists an explicit integer $N$ such that
\[\gmax^{X}(d)\leq \frac{1}{2sn}d^2+\frac{snm^3-4m}{2}d+1\]
for any $d\geq N$.
\end{theorem}

As $s\geq 1$, this also proves a stronger version of a conjectural bound of $\gmax^X$ in \cite[Example 4.4]{bayer2016space} when $d$ is sufficiently large (see also \cite[Conjecture 3.3]{tramel}).


We refer to Theorem \ref{thm-general-bound-isolated} and Corollary \ref{cor-pic-rk-1} for more general results and the definition of $N$. When $X$ is a smooth Calabi--Yau 3-fold over $\CC$, we have $s=1$, and the bound in Theorem \ref{intro-thm-pic-1} implies Theorem \ref{thm-intro-GV} (cf.~Remark \ref{rmk-determine-s}).

If $\NS(X)$ is generated by a very ample divisor, we can take $m=s=1$ and obtain the following simpler bound:

\begin{corollary}[{Corollary \ref{cor-very-ample}}]\label{intro-cor-very-ample}
Let $X$ be a factorial projective 3-fold over $\kk$ with at worst isolated singularities of Picard number one and degree $n$. Assume that $\NS(X)$ is generated by a very ample divisor. Then there exists an integer $N$ such that
\[\gmax^{X}(d)\leq \frac{1}{2n}d^2+\frac{n-4}{2}d+1\]
for any $d\geq N$. 
\end{corollary}

\begin{remark}
Note that when $X$ is smooth and $X\subset \PP^4$, the bound in Corollary \ref{intro-cor-very-ample} coincides with the one obtained in \cite{soheyla:cast,liu-ruan:cast-bound} by assuming the BMT conjecture for $X$. As the BMT conjecture is only known for smooth hypersurfaces in $\PP^4$ of degree $\leq 5$, Corollary \ref{intro-cor-very-ample} recovers and generalizes the result in \cite{soheyla:cast,liu-ruan:cast-bound} for hypersurfaces.
\end{remark}

One may wonder if it is possible to improve the bounds in Theorem \ref{intro-thm-pic-1} and Corollary \ref{intro-cor-very-ample}. The following result shows that these bounds are already asymptotically optimal and completely describe the asymptotic behavior of $\gmax^{X}(d)$.

\begin{theorem}[{Theorem \ref{thm-limit}}]\label{intro-thm-limit}
Let $X$ be a factorial projective 3-fold over $\kk$ with at worst isolated singularities of Picard number one and degree $n$, and $H$ be the ample generator of $\NS(X)$. Let $s$ be the least integer such that $|sH|\neq \varnothing$.  Then
\[\lim_{d\to +\infty} \frac{\gmax^{X}(d)}{d^2}=\frac{1}{2sn}.\]
\end{theorem}


For the generalization of the above results to higher-dimensional projective varieties, see Section \ref{subsec-high-dim}.

\subsection{Vanishing of surface-counting invariants of Calabi--Yau 4-folds}\label{intro-subsec-4fold}

Recently, the enumerative geometry of Calabi--Yau 4-folds has attracted much attention. Analogous to curve-counting theories of 3-folds, several surface-counting theories of Calabi--Yau 4-folds are built in \cite{bae-kool-park:count-surface-I,bae-kool-park:count-surface-II,joyce:vfc-4fold,oh-thomas:I}, such as (reduced) Donaldson--Thomas invariants and Pandharipande--Thomas invariants. 

More precisely, we assume that $X$ is a projective Calabi--Yau 4-fold and fix a class 
\[v=\big(0,0,\gamma,\beta, n-\gamma.\td_2(X)\big)\in \mathrm{H}^*(X, \QQ).\]
Following the notations in \cite{bae-kool-park:count-surface-I}, for any $q\in \{-1,0,1\}$, let $\mathscr{P}_v^{(q)}(X)$ be the corresponding moduli space of $\PT_q$-stable pair on $X$ (cf.~\cite[Definition 2.1]{bae-kool-park:count-surface-I}). Note that $\mathscr{P}_v^{(-1)}(X)$ parameterizes closed subschemes $Z\subset X$ such that $\ch(\oh_Z)=v$.

There are several different surface-counting invariants defined over the moduli spaces $\mathscr{P}_v^{(q)}(X)$. By \cite{joyce:vfc-4fold, oh-thomas:I} and \cite[Theorem 1.4]{bae-kool-park:count-surface-I}, there is a virtual cycle on $\mathscr{P}_v^{(q)}(X)$ for each $q\in \{-1,0,1\}$. Thus, we can define surface-counting invariants over $\mathscr{P}_v^{(q)}(X)$. On the other hand, \cite{bae-kool-park:count-surface-I} constructs a reduced virtual cycle on $\mathscr{P}_v^{(q)}(X)$ for each $q\in \{-1,0,1\}$, hence we can also define the corresponding reduced surface-counting invariants.

In this paper, we prove an effective vanishing theorem for any enumerative invariant defined over $\mathscr{P}_v^{(q)}(X)$:

\begin{corollary}[{Corollary \ref{cor-vanish-cy4}}]
Let $X$ be projective Calabi--Yau 4-fold of Picard number one and degree $n$. Assume that $\NS(X)=\ZZ H$ with $H$ very ample. Then there exists an explicit integer $N$, such that for any class
\[v=\big(0,0,\gamma,\beta, n-\gamma.\td_2(X)\big)\in \mathrm{H}^*(X, \QQ),\]
if $d:=\gamma.H^2\geq N$ and 
\[\beta.H<-(\frac{1}{2n}d^2+\frac{n-5}{2}d),\]
then
\[\mathscr{P}_v^{(q)}(X)=\varnothing\]
for any $q\in \{-1,0,1\}$.
\end{corollary}

See Corollary \ref{cor-vanish-cy4} for the definition of $N$ and a sharper result. The idea is to intersect surfaces with general hyperplane sections of 4-folds and apply Corollary \ref{intro-thm-pic-1} to these hyperplane sections, which are general type 3-folds. Currently, such a result can not be proven by applying the BMT conjecture and the wall-crossing on hyperplane sections, since there are no known examples of general type 3-folds satisfying the BMT conjecture.


\subsection{Low-degree curves}\label{intro-subsec-low-deg}



As another application of our method, we study low-degree curves on complete intersection Calabi--Yau 3-folds in projective spaces. In \cite{soheyla:cast,HKQ09}, GV-invariants of these 3-folds are calculated up to a high genus via physical methods. An observation is that, the original Conjecture \ref{conj-cast} is not optimal, but we can modify it slightly to obtain a (conjectural) optimal bound, which is Conjecture \ref{conj-optimal-bound}. 

\begin{conjecture}[{Conjecture \ref{conj-optimal-bound}}]
Let $X$ be a smooth complete intersection Calabi--Yau 3-fold of (total) degree $n$ in a projective space. Then
\[g_{\max}^{X}(d)\leq \frac{1}{2n}d^2+\frac{1}{2}d+1-\epsilon_X(d),\]
where $\epsilon_X(d)$ is an explicit function defined in Conjecture \ref{conj-optimal-bound}.
\end{conjecture}

Based on our techniques described in Section \ref{intro-subsec-method}, we prove Conjecture \ref{conj-optimal-bound} for low-degree curves, which will be the starting point of studying Conjecture \ref{conj-optimal-bound} by assuming the BMT conjecture in a sequel \cite{fey-liu:sharp}. 


\begin{theorem}[{Theorem \ref{thm-low-deg}}]
Let $X=X_{k_1,\cdots, k_r}$ be a factorial complete intersection 3-fold in $\PP^{r+3}$ of multi-degree $(k_1,k_2,\dots, k_r)$, where $2\leq k_1\leq \cdots\leq k_r$. Then Conjecture \ref{conj-optimal-bound} holds for $d\leq D_1$, where

\begin{enumerate}
    \item $X_5$: $D_1=15$,

    \item $X_{2,4}$: $D_1=8$,

    \item $X_{3,3}$: $D_1=9$,

    \item $X_{2,2,3}$: $D_1=6$, and

    \item $X_{2,2,2,2}$: $D_1=6$.
\end{enumerate}

\end{theorem}

We also consider Pfaffian--Grassmannian Calabi--Yau 3-folds in Section \ref{subsec-grass}. See Remark \ref{rmk-weak-bound} for more results along this direction.

In \cite{fey-liu:sharp}, we will refine the above results and study curves in other 3-folds via the method in Section \ref{intro-subsec-method}, e.g.~complete intersection Calabi--Yau 3-folds in weighted projective spaces.

\subsection{An iterative method}\label{intro-subsec-method}



Now we describe our method, which is an iterative algorithm obtained by combining classical techniques with the wall-crossing method and works for any projective 3-folds. For simplicity, we start with a polarised smooth projective 3-fold $(X, H)$ over $\kk$ with $mH$ very ample and $C\subset X$ be a $1$-dimensional closed subscheme of degree $d$.

\bigskip

\textbf{Step $1$: Projection}

Although the BMT Conjecture for $X$ may not be known, we can translate the problem of bounding $g(C)$ to the problem of bounding $g(C')$ for a new $1$-dimensional closed subscheme $C'$ in a smooth projective 3-fold $X'$ satisfying the BMT conjecture.

More precisely,  we can find a smooth projective 3-fold $X'$ with a rational map $\pi_X\colon X\dashrightarrow X'$ such that $\pi_C:=\pi_X|_C \colon C\to C':=\pi_X(C)$ is a birational morphism and $X'$ satisfies BMT Conjecture. As it is easy to see $g(C)\leq g(C')$, we only need to bound $g(C')$, and the BMT conjecture for $X'$ provides a powerful tool to achieve this.

Such $X'$ and $\pi_X$ always exist. Indeed, we can embed $X$ into a projective space $\PP^n$ via $mH$, then projections from general points (cf.~Lemma \ref{lem-good-proj-exist}) give the desired $\pi_X$. In this case, we take $X'=\PP^3$ which satisfies BMT Conjecture by Theorem \ref{SBG}, and $\pi_X$ is the restriction of the composition of projections from general points to $X$. For the other choices of $X'$ and $\pi_X$, see Remark \ref{rmk-step-1}.

\bigskip

\textbf{Step $2$: Wall-crossing}

Next, we need to do wall-crossing on $X'$. By Lemma \ref{bms lemma 2.7}, $I_{C'/X'}$ is $\sigma_{a,b}$-stable for any $b<0$ and $a\gg 0$, where $\sigma_{a,b}$ is the tilt-stability (cf.~Section \ref{subsec-tilt}). Then we can move down $(a,b)$ and analyze possible walls that it may meet. A key observation is that the existence of walls for $I_{C'/X'}$ is closely related to the geometry of $C'$. More precisely, by Proposition \ref{cor_wall} and Lemma \ref{lem-construct-wall}, the existence of a wall of $I_{C'/X'}$ in the range $\{a>0\}\times \{b_d<b<0\}$ is equivalent to the existence of a divisor $S'\subset X'$ such that $\deg(C'\cap S')$ is large than an integer determined by the location of the wall.


By pulling back $S'$ and $C'$ to $X$ via the projection $\pi_X$, we have the following two situations:

\begin{enumerate}[{(i)}]
    \item there is a divisor $S$ on $X$ containing $C$, or 

    \item there is a divisor $S$ on $X$ such that $\deg(C\cap S)$ is bigger than a given integer, and there is another $1$-dimensional closed subscheme $C_1\subset C$ such that $\deg(C_1)+\deg(C\cap S)=\deg(C)$ with
    \[g(C)=g(C_1)+g(C\cap S)+C_1.S-1.\]
\end{enumerate}


\bigskip

\textbf{Step $3$: Variation of projections}

Although we only need to study $g(C')$ in order to bound $g(C)$, the information about the geometry may be lost when we project $C$ to $C'$, so the upper bound of $g(C)$ given by $g(C')$ may not fit our expectations. Fortunately, thanks to the flexibility of the choice of the projection $\pi_X\colon X\dashrightarrow X'$, we can vary it to find a better projection as possible. Moreover, the wall-crossing of $I_{C'/X'}$ in Step 2 for each projection would bring new restrictions on the geometry of $C\subset X$ and we will explain as follows. To explain this step, we take $X'=\PP^3$ for simplicity.

\begin{itemize}

    \item In Step 2, if there is a wall for $I_{C'/X'}$ in $\{a>0\}\times \{b_d<b<0\}$, then we can find a divisor $S$ in $X$ either containing $C$ or $\deg(C\cap S)$ is larger than a given integer. In this case, we can vary $\pi_X$ to find a new projection $\pi'_X\colon X\to \PP^3$ such that $\pi'_X$ not only maps $C$ birationally onto its image $C''$ but also maps $S$ birationally onto its image $S''$. Therefore, compared to the image $C'=\pi_X(C)$, the curve $C''$ carries a new structure as it is related to a surface $S''\subset \PP^3$. Thus, if we go back to Step 2 and do wall-crossing for $I_{C''/\PP^3}$, we can obtain further restriction on $g(C)$.

    In some situations, we may end up at case (ii) in Step 2 with some divisors in $X$ around $C$ after processing these three steps finitely many times. Then it reduces the question on $C$ to lower-degree curves, and an induction argument on the degree together with the restrictions given by divisors in $X$ may lead to a desired bound on $g(C)$. 

    \item In Step 2, if there is no wall for $I_{C'/X'}$ in the range $\{a>0\}\times \{b_d<b<0\}$, we get an upper bound of $g(C)$ via bounding $g(C')$ by applying the BMT conjecture to $(a,b)=(0,b_d)$ and $I_{C'/X'}$. However, this bound can be refined, as the inequality $g(C)\leq g(C')$ is strict in most cases. 
    
    For example, if $g(C)=g(C')$, then this implies that the projection $\pi_X|_C$ maps isomorphically onto $C'$. Hence $C\subset X\cap \PP^3 \subset \PP^n$, which gives a strong restriction on the geometry of $C$. In general, if $g(C)\geq g(C')-t$ for $0\leq t\leq n-4$, then $C\subset \PP^{3+t}\cap X$ (see e.g.~Lemma \ref{lem-diff-g}). In this case, we can vary $\pi_X$ to find a new projection $\pi'_X\colon X\to \PP^3$ such that the corresponding map $\PP^n\dashrightarrow \PP^{3+t}$ is a closed immersion when restricted to $C$. This reduces $C$ to a curve in a lower-dimensional projective space and simplifies the geometry of $C$.

\end{itemize}

\bigskip


Processing the above three steps each time adds new restrictions on $C$. We can iterate the above steps arbitrarily finitely many times and obtain a family of projections and divisors in $X$ related to $C$. Combining these restrictions together, we can get effective bounds of $g(C)$. These techniques will be used repeatedly in Section \ref{sec_bound} and Section \ref{sec-low-deg}.

\begin{remark}\label{rmk-step-1}
There are other choices of $X'$ and $\pi_X$, depending on the explicit geometry of $X$. For example, if $X\subset \PP(1,1,1,1,2)$ is a hypersurface in the weighted projective space, we can also perform the projection from a general point and map $C$ birationally into $\PP(1,1,1,2)$. In this case, we take $X'$ to be the blow-up of the unique singular point in $\PP(1,1,1,2)$, which is a toric 3-fold, and its BMT conjecture is studied in \cite{zhao:stability-fano-threefold}. In \cite{fey-liu:sharp}, we will apply this construction to study curves in complete intersection 3-folds in weighted projective spaces. However, to obtain general results for a large class of projective 3-folds as in Section \ref{intro-bound-genus}, $X'=\PP^3$ is the best and also the universal choice.
\end{remark}

\begin{remark}
Note that the above method actually proves a weak version of the BMT conjecture for some rank one objects. It is an interesting question if we can extend the method above to prove a version of BMT conjecture for $X$, which allows us to construct stability conditions on $\D^b(X)$ as in \cite{bayer2011bridgeland,bayer2016space} (see also \cite[Section 4.1]{bayer:ICM}). We will back to this point in future work.
\end{remark}

\begin{remark}
Although one of our main results (Theorem \ref{thm-general-bound-isolated}) works without the assumption of the Picard number, it is more natural to prove a bound of the genus involving all primitive polarizations at once. However, we do not know a conjectural picture for 3-folds of higher Picard numbers. It is very interesting to study curves in these 3-folds, such as elliptic Calabi--Yau 3-folds. We will apply our method to these examples in a sequel \cite{guo-liu-ruan}.
\end{remark}

\subsection*{Plan of the paper}

First, we review some basic definitions and properties of tilt-stability in Section \ref{sec:stability_condition}. Then in Section \ref{sec-proj}, we investigate the behavior of curves under projections between projective spaces and relate them to the wall-crossing of ideal sheaves with respect to tilt-stability. Next, using results in previous sections, we can run our algorithm in Section \ref{intro-subsec-method} and prove the bounds of the genus of curves in a fixed 3-fold in Section \ref{sec_bound}. In particular, we deduce Theorem \ref{intro-thm-pic-1}, Corollary \ref{intro-cor-very-ample}, and Theorem \ref{intro-thm-limit} in Section \ref{subsec-pic-1}. In Section \ref{sec-low-deg}, we continue to use the method in Section \ref{intro-subsec-method} to study low-degree curves on complete intersection Calabi--Yau 3-folds and Pfaffian--Grassmannian Calabi--Yau 3-folds, and prove a refined version of Conjecture \ref{conj-cast} in low-degree (cf.~Theorem \ref{thm-low-deg}). Finally, we discuss some applications of our main results in Section \ref{sec_app}, such as vanishing theorems of enumerative invariants (cf.~Theorem \ref{thm-vanishing} and Corollary \ref{cor-vanish-cy4}).

\subsection*{Acknowledgements} 

I would like to thank my supervisor Yongbin Ruan for many valuable suggestions. I also would like to thank Arend Bayer, Soheyla Feyzbakhsh, Shuai Guo, Yong Hu, Chen Jiang, Qingyuan Jiang, Chunyi Li, Emanuele Macr\`i, Songtao Ma, Yukinobu Toda, Chenyang Xu, Zhankuan Yin, and Xiaolei Zhao for useful conversations.

\subsection*{Notations and conventions} \leavevmode

\begin{itemize}
    \item Throughout this paper, we work over an algebraically closed field $\kk$ of any characteristic except in Section \ref{subsec-vanish} where we set $\kk=\CC$. All schemes are assumed to be finite type over $\kk$.  


    \item A divisor always means a Weil divisor. For a line bundle or Cartier divisor $\mathcal{L}$, the associated linear series is denoted by $|\mathcal{L}|$. 

    \item Let $X$ be a projective scheme and $D$ be a Cartier divisor on $X$. For any closed subscheme $Y$ in $X$, we denote by $Y.D^n$ the cycle class $c_1(\oh_X(D))^n\cap Y$ in the Chow group of $X$ for any positive integer $n$. If $\dim Y=n$, then we also denote by $Y.D^n$ the integer number $\int_X Y.D^n=\int_X c_1(\oh_X(D))^n\cap Y$. If $Y\subset X$ is a closed subscheme of a projective variety $X$, then for any ample divisor $H$ on $X$, we define the \emph{$H$-degree} of $Y$ by $\deg_H(Y):=Y.H^{\dim Y}$. 
    

    \item We say a $1$-dimensional projective scheme $C$ is a \emph{curve} if it is Cohen--Macaulay. This is equivalent to saying that $C$ has no isolated or embedded points, or in other words, $\oh_C$ is a pure sheaf. The (arithmetic) genus of $C$ is denoted by $g(C):=1-\chi(\oh_C)$.

    \item Let $X$ be a projective scheme. If $\Pic(X)\cong \ZZ$, we denote by $\oh_X(1)$ the ample generator of $\Pic(X)$. The corresponding divisor is denoted by $H:=c_1(\oh_X(1))$. In this case, 
    
    \item For any projective scheme $X$, we denote by $\NS(X)$ the Neron--Severi group of $X$. When $\NS(X)=\ZZ H$ where $H$ is ample, the $H$-degree of a $1$-dimensional closed subscheme $C\subset X$ is denoted by $\deg(C)$ and called the degree of $C$ for simplicity.

    \item We denote by $I_{Y/X}$ the ideal sheaf of a closed subscheme $Y$ in a scheme $X$. If $X$ is a given 3-fold, then we set $I_{Y}:=I_{Y/X}$ for simplicity.

    \item If $f\colon X\to Y$ is a morphism between projective schemes, we denote the schematic image of $f$ by $f(X)$, which is a closed subscheme of $Y$. We say $f$ is a birational morphism if there exist open dense subschemes $U\subset X$ and $V\subset Y$ such that $f(U)\subset V$ and $f|_U\colon U\to V$ is an isomorphism.

    \item For a perfect complex $E$ on a projective scheme $X$, we denote the $i$-th Chern character of $E$ by $\ch_i(E)$. For an integer $n\geq 0$, the $n$-truncated Chern character is defined by
    \[\ch_{\leq n}(E):=\big(\ch_0(E), \ch_1(E),\dots, \ch_n(E)\big).\]
For a real number $b\in \mathbb{R}$ and a Cartier divisor $H$ on $X$, the $b$-twisted Chern character is
    \[\ch^{b H}(E):=\exp(-b H).\ch(E).\]
We will always take $H$ to be a fixed ample divisor and write $\ch^{b}_i(E):=\ch^{bH}_i(E)$ for simplicity. For any integer $n\geq 0$, we define 
\[\ch_{H, \leq n}(E):=\big(\ch_{H,0}(E), \ch_{H,1}(E),\dots, \ch_{H,n}(E)\big)\in \QQ^{n+1},\]
where
\[\ch_{H, i}(E):=\frac{H^{\dim X-i}\ch_i(E)}{H^{\dim X}}\in \mathbb{Q}.\]
\end{itemize}



\section{Tilt-stability} \label{sec:stability_condition}

Let $X$ be a smooth projective variety over $\kk$ and $\D^b(X)$ be the bounded derived category of coherent sheaves on $X$.
In this section, we recall the construction of weak stability conditions on $\D^b(X)$ and the notion of tilt-stability. 


We denote by $\KK(X)$ the K-group of $\D^b(X)$. We fix a surjective morphism $v \colon \KK(X) \twoheadrightarrow \Lambda$ to a finite rank lattice. 



\begin{definition}
A \emph{weak stability condition} on $\D^b(X)$ is a pair $\sigma = (\cA, Z)$ where $\cA$ is the heart of a bounded t-structure on $\D^b(X)$ and $Z \colon \Lambda \ra \CC$ is a group homomorphism such that 
\begin{enumerate}[(i)]
    \item the composition $Z \circ v \colon \KK(\cA) \cong \KK(X) \ra \CC$ satisfies that for any non-zero $E \in \cA$ we have $\Im Z(E) \geq 0$, and if $\Im Z(E) = 0$ then $\Re Z(E) \leq 0$. From now on, we write $Z(E)$ rather than $Z(v(E))$.
\end{enumerate}
For any $E \in \cA$, we define the \emph{slope} of $E$ with respect to $\sigma$ as
\[
\mu_\sigma(E) := \begin{cases}  - \frac{\Re Z(E)}{\Im Z(E)}, & \text{if} ~ \Im Z(E) > 0 \\
+ \infty , & \text{otherwise}.
\end{cases}
\]
We say an object $0 \neq E \in \cA$ is $\sigma$-(semi)stable if $\mu_\sigma(F) < \mu_\sigma(E/F)$ (respectively, $\mu_\sigma(F) \leq \mu_\sigma(E/F)$) for all proper subobjects $F \sst E$. 
\begin{enumerate}[(i), resume]
    \item Any object $E \in \cA$ has a Harder--Narasimhan filtration in terms of $\sigma$-semistability defined above.
    \item There exists a quadratic form $Q$ on $\Lambda \otimes_{\ZZ} \mathbb{R}$ such that $Q|_{\ker Z}$ is negative definite, and $Q(E) \geq 0$ for all $\sigma$-semistable objects $E \in \cA$. This is known as the \emph{support property}.
\end{enumerate}
\end{definition}


\subsection{Tilt-stability}\label{subsec-tilt}
Let $(X, H)$ be a $n$-dimensional polarised smooth projective variety. Starting with the classical slope stability, where  
\[
\mu_H(E) := \begin{cases}  \frac{H^{n-1}\ch_1(E)}{H^n\ch_0(E)}, & \text{if} ~ \ch_0(E)\neq 0 \\
+ \infty , & \text{otherwise}.
\end{cases}
\]
for any $E\neq 0\in \Coh(X)$, we can form the once-tilted heart $\cA^{b}$ for any $b\in \mathbb{R}$ as follows.
For the slope we just defined, every sheaf $E$ has a Harder--Narasimhan filtration. Its graded pieces have slopes whose maximum we denote by $\mu^+_H(E)$ and minimum by $\mu^-_H(E)$. Then for any $b\in \mathbb{R}$, we define an abelian category $\cA^b\subset \D^b(X)$ as
\[\cA^b:=\{E\in \D^b(X) \mid \mu^+_H(\cH^{-1}(E))\leq b,~ b<\mu^-_H(\cH^0(E)),~\text{and}~\cH^i(E)=0~\text{for}~i\neq 0, -1\}.\]

It is a result of \cite{happel1996tilting} that the abelian category $\cA^b$ is the heart of a bounded t-structure on $\D^b(X)$ for any $b\in \mathbb{R}$.

Now for $E \in \cA^{b}$, we define
\[ Z_{a, b}(E) := \frac{1}{2} a^2 H^n \ch_0^{b H}(E) - H^{n-2} \ch_2^{b H}(E) + \mathfrak{i} H^{n-1} \ch_1^{b H}(E). \]

\begin{proposition}[{\cite{bayer2011bridgeland, bayer2016space}}] \label{tilt_stab}
Let $a>0$ and $b \in \mathbb{R}$. Then the pair $\sigma_{a, b} = (\cA^{b} , Z_{a, b})$ defines a weak stability condition on $\D^b(X)$. The quadratic form $Q$ is given by the discriminant
\begin{equation} \label{BG}
    \Delta_H(E) = \big(H^{n-1} \ch_1(E)\big)^2 - 2 H^n \ch_0(E) H^{n-2} \ch_2(E).
\end{equation}
The weak stability conditions $\sigma_{a, b}$ vary continuously as $(a, b) \in \mathbb{R}_{>0} \times \mathbb{R}$ varies. 
\end{proposition}

The weak stability conditions defined above are called \emph{tilt-stability conditions}. We denote the slope function of $\sigma_{a, b}$ by $\mu_{a, b}(-)$.


\begin{lemma} [{\cite[Lemma 2.7]{bayer2016space}, \cite[Proposition 4.8]{bayer2020desingularization}}] \label{bms lemma 2.7}
Let $E \in \D^b(X)$ be an object and $\mu_H(E)>b$. Then $E\in \cA^{b}$ and $E$ is $\sigma_{a, b}$-(semi)stable for $a \gg 0$ if and only if $E$ is a 2-Gieseker-(semi)stable sheaf.


\end{lemma}



\subsection{Bayer--Macr\'i--Toda conjecture} 

There is another Bogomolov--Gieseker-type inequality, which is conjectured by \cite{bayer2011bridgeland, bayer2016space} in order to construct stability conditions on $\D^b(X)$. We call it \emph{Bayer--Macr\'i--Toda (BMT) conjecture}.

\begin{conjecture} [{\cite[Conjecture 4.1]{bayer2016space}}] \label{SBG_conj}
Let $(X, H)$ be a polarised smooth projective 3-fold. Assume that $E$ is a $\sigma_{a, b}$-semistable object. Then
\[Q_{a, b}(E):=a^2 \Delta_H(E)+4\big(H \ch_2^{bH}(E)\big)^2-6\big(H^2\ch_1^{bH}(E)\big) \ch_3^{bH}(E)\geq 0.\]
\end{conjecture}

This conjecture is proven in several cases and we refer to \cite[Section 4]{bayer:ICM} for a list of known results. In our paper, we only need the following result for $\PP^3$.

\begin{theorem}[{\cite{Mac14}}]\label{SBG}
Conjecture \ref{SBG_conj} holds for $\PP^3$ over any algebraically closed field $\kk$.

\end{theorem}



\begin{remark} \label{rmk_boundary}
If an object $E\in \D^b(X)$ satisfies $Q_{a,b}(E)\geq 0$ for any $(a,b)\in U$, where $U\subset \mathbb{R}_{>0}\times \mathbb{R}$ is a given subset, then by the continuity, we have $Q_{a,b}(E)\geq 0$ for any $(a,b)\in \overline{U}$.
\end{remark}


\subsection{Wall-chamber structure} \label{wall_sec2}

Let $(X, H)$ be a polarised smooth projective 3-fold. In this subsection, we describe the wall-chamber structure of tilt-stability $\sigma_{a, b}$.

\begin{definition}
Let $v\in \KK(X)$. A \emph{numerical wall} for $v$ induced by $w\in \KK(X)$ is the subset
\[W(v, w)=\{(a, b)\in \mathbb{R}_{>0}\times \mathbb{R} \mid \mu_{a, b}(v)=\mu_{a, b}(w)\}\neq \varnothing.\]

Let $E\in \D^b(X)$ be a non-zero object with $[E]=v\in \mathrm{K}(X)$. A numerical wall $W=W(v,w)$ for $v$ is called an \emph{actual wall for $E$} if there is a short exact sequence of $\sigma_{a,b}$-semistable objects
\begin{equation} \label{wall_seq}
    0\to F\to E\to G\to 0
\end{equation}
in $\cA^{b}$ for one $(a,b)\in W$ (hence all $(a,b)\in W$ by \cite[Corollary  4.13]{bayer2020desingularization}) such that $W=W(E, F)$ and $[F]=w$.

We say such an exact sequence  \eqref{wall_seq} of $\sigma_{a,b}$-semistable objects  \emph{induces the actual wall $W$} for $E$.


A \emph{chamber} for $E$ is defined to be a connected component of the complement of the union of actual walls for $E$.

We say a point $(a,b)\in \mathbb{R}_{>0}\times \mathbb{R}$ is \emph{over} a numerical wall $W$ if $a>a'$ for any $(a',b)\in W$.
\end{definition}



\begin{proposition} [{\cite[Section 4]{bayer2020desingularization}}] \label{wall_prop}
Let $v\in \KK(X)$ and $\Delta_H(v)\geq 0$. Let $E\in \D^b(X)$ be an object with $[E]=v$.

\begin{enumerate}
    \item If $\mathcal{C}$ is a chamber for an object $E$, then $E$ is $\sigma_{a,b}$-semistable for some $(a,b)\in \cC$ if and only if it is for all $(a,b)\in \cC$.
    
    \item A numerical wall for $v$ is either a semicircle centered along the $b$-axis or a vertical wall parallel to the $a$-axis. The set of numerical walls for $v$ is locally finite. No two walls intersect.
    
    \item If $\ch_0(v)\neq 0$, then there is a unique numerical vertical wall $b=\mu_H(v)$. The remaining numerical walls are split into two sets of nested semicircles whose apex lies on the hyperbola $\mu_{a, b}(v)=0$.
    
    \item If $\ch_0(v)=0$ and $H^2\cdot \ch_1(v)\neq 0$, then every numerical wall is a semicircle whose apex lies on the ray $b=\frac{H\cdot \ch_2(v)}{H^2\cdot \ch_1(v)}$.
    
    \item If $\ch_0(v)=0$ and $H^2\cdot \ch_1(v)= 0$, then there are no actual walls for $v$.
    
    \item If an actual wall for $E$ is induced by a short exact sequence of $\sigma_{a,b}$-semistable objects
    \[0\to F\to E\to G\to 0\]
    then $\Delta_H(F)+\Delta_H(G)\leq \Delta_H(E)$ and equality can only occur if either $F$ or $G$ is a sheaf supported in dimension zero. 
\end{enumerate}

\end{proposition}

By Proposition \ref{wall_prop}, an actual wall gives two adjacent chambers. Walls and chambers can be visualized as in \cite[Figure 1]{bayer2020desingularization}.

\begin{definition}\label{def-uppermost}
Let $E$ be a 2-Gieseker-semistable sheaf on $X$. We say a semicircle actual wall $W$ is \emph{the uppermost wall} for $E$ if there is no other actual wall over $W$ and $W\subset \{0<a\}\times \{b<\mu_H(E)\}$.
\end{definition}

\section{Projection}\label{sec-proj}

\subsection{Walls}\label{wall_sec}

In this section, we are going to relate the wall-chamber structure for the ideal sheaf of a curve to its geometry. We begin with a standard lemma.

\begin{lemma}\label{lem-decompose}
Let $X$ be a smooth projective variety with $\dim X\geq 3$. Let $C\subset X$ be a closed subscheme with $\dim C\leq \dim X-2$ and $D\subset X$ be a divisor. If $C\neq C\cap D$, then we have an exact sequence
$0\to \oh_{C_1}(-D)\to \oh_C\to \oh_{C\cap D}\to 0$,
where $C_1\subset C$ is a closed subscheme.
\end{lemma}

\begin{proof}
Consider the following commutative diagram
\[\begin{tikzcd}
	{\oh_X} & {\oh_C} \\
	{\oh_D} & {\oh_{C\cap D}}
	\arrow[two heads, from=1-2, to=2-2]
	\arrow[two heads, from=1-1, to=2-1]
	\arrow[two heads, from=2-1, to=2-2]
	\arrow[two heads, from=1-1, to=1-2]
\end{tikzcd}\]
where morphisms are natural surjections. By taking kernels, we have a commutative diagram
\[\begin{tikzcd}
	0 & {I_{C}} & {\oh_X} & {\oh_C} & 0 \\
	0 & {I_{C\cap D/D}} & {\oh_D} & {\oh_{C\cap D}} & 0
	\arrow[two heads, from=1-4, to=2-4]
	\arrow[two heads, from=1-3, to=2-3]
	\arrow[from=2-3, to=2-4]
	\arrow[from=1-3, to=1-4]
	\arrow[from=1-4, to=1-5]
	\arrow[from=2-4, to=2-5]
	\arrow[from=2-1, to=2-2]
	\arrow[from=1-1, to=1-2]
	\arrow[from=1-2, to=1-3]
	\arrow[from=2-2, to=2-3]
	\arrow["s"', from=1-2, to=2-2]
\end{tikzcd}\]
with exact rows. Note that
\[I_{C\cap D/D}=\frac{I_{C\cap D}}{I_D}=\frac{I_C+I_D}{I_D}=\frac{I_C}{I_C\cap I_D},\]
hence $s$ is surjective with $\ker(s)=I_C\cap I_D$. Since $\dim C\leq \dim X-2$, we see $c_1(I_{C\cap D/D})=D$ and $\ker(s)$ is a rank one torsion-free sheaf with $c_1(\ker(s))=-D$. As $X$ is smooth, the double dual $(\ker(s))^{\vee \vee}$ is a line bundle on $X$ by \cite[Proposition 1.9]{har80}. Since $c_1(\ker(s))=-D$, we have $(\ker(s))^{\vee \vee}\cong \oh_X(-D)$ and $\ker(s)=I_{C_1}(-D)$, where $C_1\subset X$ is a closed subscheme. Therefore, the statement follows from the Snake Lemma.
\end{proof}

Let $X$ be a smooth projective 3-fold such that $\Pic(X)$ is generated by an ample line bundle $\oh_X(1)$. Let $H:=c_1(\oh_X(1))$ and $n:=H^3$ be the degree of $X$. 

For any $1$-dimensional closed subscheme $C\subset X$ of degree $d$, we have $\ch_{H,\leq 2}(I_C)=(1,0,-\frac{d}{n})$. By Lemma \ref{bms lemma 2.7}, $I_C\in \cA^b$ is $\sigma_{a,b}$-stable for any $b<0$ and $a\gg 0$. We define \[\rho_d:=\sqrt{\frac{d}{4n}},\quad b_d:=\rho_d-\sqrt{\rho_d^2+\frac{2d}{n}}=-\sqrt{\frac{d}{n}}.\] 

\begin{proposition} \label{cor_wall}
Let $C\subset X$ be a curve with degree $d$. Assume that for some $b_d\leq b<0$ there is an actual wall for $I_C$ given by an exact sequence in $\cA^{b}$
\[0\to A\to I_C\to B\to 0.\]

\begin{enumerate}
    \item We have $b_d\leq b<-1$ and $A= \oh_X(-D)$ or $A=I_{C_1}(-D)$ for a divisor $D\in |\oh_X(k)|$, where $\lceil b \rceil \leq -k\leq -1$ and $C_1\subset C$ is a curve of degree
\[d_1:=\deg(C_1)< \min\{d-\frac{k^2 n}{2}, d+\frac{k^2 n}{2}-k\sqrt{2nd}\}.\]

\item In the case $A=I_{C_1}(-D)$, there is another one-dimensional closed subscheme $C_2\subset C\cap D$ of degree $d_2$ such that $d_1+d_2=d$ and 
\begin{equation} \label{key_equation}
    g(C)=g(C_1)+g(C_2)+kd_1-1.
\end{equation}

\item If this is the uppermost wall for $I_C$ and $A$ is not a line bundle, then we can take $C_2=C\cap D$ and we have an exact sequence
\begin{equation}
    0\to \oh_{C_1}(-k)\to \oh_C\to \oh_{C\cap D}\to 0.
\end{equation}
\end{enumerate}

\end{proposition}

\begin{proof}
(a) and (b) follows from \cite[Corollary 3.3]{liu-ruan:cast-bound}.

To prove (c), we have a commutative diagram as in Lemma \ref{lem-decompose} and \cite[Corollary 3.3]{liu-ruan:cast-bound}
\[\begin{tikzcd}
	& 0 & 0 & 0 \\
	0 & {I_{C_1'}(-D)} & {\oh_X(-D)} & {\oh_{C_1'}(-D)} & 0 \\
	0 & {I_C} & {\oh_X} & {\oh_C} & 0 \\
	0 & I_{C\cap D/D} & {\oh_D} & {\oh_{C\cap D}} & 0 \\
	& 0 & 0 & 0
	\arrow[from=2-1, to=2-2]
	\arrow[from=2-3, to=2-4]
	\arrow[from=2-4, to=2-5]
	\arrow[from=3-4, to=3-5]
	\arrow[from=3-2, to=3-3]
	\arrow[from=3-3, to=3-4]
	\arrow[from=3-1, to=3-2]
	\arrow[from=2-2, to=3-2]
	\arrow[from=2-3, to=3-3]
	\arrow[from=2-4, to=3-4]
	\arrow[from=2-2, to=2-3]
	\arrow[from=4-1, to=4-2]
	\arrow[from=4-2, to=4-3]
	\arrow[from=4-3, to=4-4]
	\arrow[from=4-4, to=4-5]
	\arrow[from=3-2, to=4-2]
	\arrow[from=3-3, to=4-3]
	\arrow[from=3-4, to=4-4]
	\arrow[from=4-4, to=5-4]
	\arrow[from=4-3, to=5-3]
	\arrow[from=4-2, to=5-2]
	\arrow[from=1-2, to=2-2]
	\arrow[from=1-3, to=2-3]
	\arrow[from=1-4, to=2-4]
\end{tikzcd}\]
with all rows and columns being exact, where $C_1'\subset C$ is a curve. Since the wall $W(I_{C_1}(-D), I_C)$ is the uppermost one, we see $\ch_{H, 2}(I_{C_1}(-D))\geq \ch_{H, 2}(I_{C_1'}(-D))$, otherwise $W(I_{C_1'}(-D), I_C)$ is an actual wall for $I_C$ over $W(I_{C_1}(-D), I_C)$. But from $C_2\subset C\cap D$, we also have a reverse inequality $\ch_{H, 2}(I_{C_1}(-D))\leq \ch_{H, 2}(I_{C_1'}(-D))$ and get
\[\ch_{H, \leq 2}(I_{C_1}(-D)) = \ch_{H, \leq 2}(I_{C_1'}(-D)).\]
In other word, $W(I_{C_1}(-D), I_C)=W(I_{C_1'}(-D), I_C)$. Thus, we can take $C_1=C_1'$ and $C_2=C\cap D$ without changing the location of the wall.
\end{proof}

\begin{lemma}\label{lem-construct-wall}
Let $C\subset X$ be a curve of degree $d$. Assume that there is a divisor $D\in |\oh_X(k)|$ for an integer $1\leq k\leq -b_d$ satisfies
\begin{equation}\label{eq-exist-wall}
    d-\deg(C\cap D)<\min\{d-\frac{k^2n}{2}, d+\frac{k^2n}{2}-k\sqrt{2nd}\}
\end{equation}
and there is no actual wall for $I_C$ over $W(I_C, I_{C_1}(-D))$, where $C_1\subset C$ is a curve defined by the natural exact sequence $0\to \oh_{C_1}(-D)\to \oh_C\to \oh_{C\cap D}\to 0$. Then the exact sequence
\begin{equation}\label{eq-sec-3-1}
    0\to I_{C_1}(-D)\to I_C\to I_{C\cap D/D}\to 0
\end{equation}
gives the uppermost actual wall of $I_C$.
\end{lemma}

\begin{proof}
From \eqref{eq-exist-wall}, we see $W(I_C, I_{C_1}(-D))\neq \varnothing$ is a semicircle numerical wall for $I_C$. As $I_{C/D}$ is a torsion sheaf, we have $I_{C/D}\in \cA^b$. Since $I_{C_1}(-D), I_C\in \cA^b$, the exact sequence \eqref{eq-sec-3-1} is also an exact sequence in $\cA^b$. Thus, it is the uppermost actual wall for $I_C$ by \eqref{eq-sec-3-1}, as there is no other actual wall for $I_C$ over $W(I_C, I_{C_1}(-D))$ by our assumption.
\end{proof}


Combining Proposition \ref{cor_wall} with Lemma \ref{lem-construct-wall}, we know that the existence of an actual wall for $I_C$ in the range $b_d\leq b<0$ is equivalent to the existence of a suitable divisor $D$ such that $\deg(C\cap D)$ is greater than a specific number (cf.~\eqref{eq-exist-wall}). This observation plays an important role in the arguments in later sections.

For convenience, we introduce a special class of curves.

\begin{definition} \label{def_good}
Let $C\subset X$ be a curve of degree $d$. For a real number $t\leq -1$, we say $C$ is \emph{$(t)$-neutral (in $X$)} if there is no actual wall for $I_C$ in the range $a>0$ and $t<b<0$.
\end{definition}


\subsection{Projection from a point}

In this section, we review classical results on the projection map from a point.

Let $V\cong \kk^{n+1}$ be a vector space of dimension $n+1\geq 2$. Then any closed point $p\in \PP(V)$ corresponds to a surjective linear map $V\to V_p$ of vector spaces with $\dim_{\kk} V_p=1$. We define $W_p:=\ker(V\to V_p)$. As in \cite[\href{https://stacks.math.columbia.edu/tag/0B1N}{Tag 0B1N}]{stacks-project}, \emph{the projection from the point $p$} is the morphism
\[r_p\colon \PP(V)\setminus p\to \PP(W_p)\]
induced by the natural inclusion $W_p\hookrightarrow V$. From the construction, we have a canonical isomorphism $$r_p^*\oh_{\PP(W_p)}(1)\cong \oh_{\PP(V)}(1)|_{\PP(V)\setminus p}.$$ Moreover, each fiber of $r_p$ is isomorphic to $\mathbb{A}^1$, and in particular, $r_p$ is smooth of relative dimension $1$.

Geometrically, $r_p(x)=r_p(y)$ for two different closed points $x,y\in \PP(V)\setminus p$ if and only if $p\in L_{x,y}$, where $L_{x,y}\cong \PP^1$ is the projective line connecting by $x$ and $y$. Moreover, for any closed point $x \in \PP(V)\setminus p$, the tangent map at $x$, denoted by $\mathrm{d}_x r_p\colon \mathrm{T}_x \PP(V)\to \mathrm{T}_{r_p(x)} \PP(W_p)$, satisfies
\[\ker(\mathrm{d}_x r_p)=\mathrm{T}_x L_{x, p}\cong \kk.\]

In the following, we will review some basic properties of projection morphisms and provide proofs for completeness.

\begin{lemma}\label{lem-proj-finite}
Let $X\subset \PP^n$ be a closed subscheme. Then $r_p|_X\colon X\to \PP^{n-1}$ is finite for any closed point $p\in \PP^n\setminus X$.
\end{lemma}

\begin{proof}
This follows from \cite[\href{https://stacks.math.columbia.edu/tag/0B1P}{Tag 0B1P}]{stacks-project} and its proof.
\end{proof}

\begin{lemma}\label{lem-04DG}
Let $f\colon X\to Y$ be a finite morphism between schemes. If $x\in X$ is a closed point and $f$ is unramified at $x$ with $f^{-1}(f(\{x\}))=\{x\}$ as sets, then there exists an open subscheme $U\subset Y$ such that $x\in f^{-1}(U)$ and $f^{-1}(U)\to U$ is a closed immersion.
\end{lemma}

\begin{proof}
Since $f$ is unramified at $x$, there is an open neighborhood $x\in W\subset X$ such that $f|_W$ is unramified. We define $V:=Y\setminus f(X\setminus W)$. Then it is easy to check that $f^{-1}(V)\subset W$. Moreover, since $f^{-1}(f(\{x\}))=\{x\}$ and $x\in W$, we know that $f(x)\notin f(X\setminus W)$ and hence $x\in f^{-1}(V)$. Thus $f^{-1}(V)\neq \varnothing$ is an open neighborhood of $x$ and $f^{-1}(V)\to V$ is finite and unramified.

By our assumption, the base field $\kk$ is algebraically closed, which gives $\kappa(x)=\kappa(f(x))=\kk$. Then the result follows from applying \cite[\href{https://stacks.math.columbia.edu/tag/04DG}{Tag 04DG}]{stacks-project} to $f^{-1}(V)\to V$.
\end{proof}

\begin{proposition}\label{prop-key-proj}
Let $X\subset \PP^n$ be a closed subscheme with $\dim X<n-1$. Fix closed points $x_1,\cdots, x_k\in X$ with $\dim \mathrm{T}_{x_i} X<n$ for any $1\leq i\leq k$. 

Then there exists an open subset $U_{x_1,\cdots,x_k}\subset \PP^n$ such that for any $p\in U_{x_1,\cdots,x_k}$, the morphism $r_p|_X\colon X\to \PP^{n-1}$ is a closed immersion over an open neighbourhood of $r_p(x_i)$ for each $1\leq i\leq k$. In other word, $X\to r_p|_X(X)$ is an isomorphism over an open neighbourhood of $r_p(x_i)$ for each $1\leq i\leq k$.
\end{proposition}

\begin{proof}
The proof is similar to \cite[Proposition 1.36]{sanna2014rational} and \cite[\href{https://stacks.math.columbia.edu/tag/0B1Q}{Tag 0B1Q}]{stacks-project}. First, we assume that $k=1$. Consider a closed subscheme $W$ of $(X\setminus x_1)\times \PP^n$ defined by
\[W:=\{(x, p)\in (X\setminus x_1)\times \PP^n~\colon~ p\in L_{x_1,x}\}.\]
Then we have two natural projection morphisms $h_1\colon W\to X\setminus x_1$ and $h_2\colon W\to \PP^n$. Note that $h_1^{-1}(x)=L_{x_1,x}\setminus x_1\cong \mathbb{A}^1$ for any $x\in X\setminus x_1$, then from \cite[\href{https://stacks.math.columbia.edu/tag/02FX}{Tag 02FX}]{stacks-project}, we see
\[\dim W\leq \dim X+1<n.\]
Hence $h_2$ is not surjective. On the other hand, consider a closed subscheme $W'$ of $\PP^n\setminus x_1$ defined by
\[W':=\{p\in \PP^n\setminus x_1~\colon~ L_{x_1, p}~\text{tangent to}~X~\text{at}~x_1\}.\]
Then we have $\dim W'=\dim \mathrm{T}_{x_1} X<n$, which implies the non-surjectivity of $h'\colon W'\to \PP^n$. Therefore, $$X\cup h_2(W)\cup h'(W')\neq \PP^n$$ by the dimension reason. As all schemes are finite type over $\kk$, the subsets $h_2(W)$ and $h'(W')$ are constructible subset in $\PP^n$ by \cite[\href{https://stacks.math.columbia.edu/tag/054K}{Tag 054K}]{stacks-project}. Hence from \cite[\href{https://stacks.math.columbia.edu/tag/0AAW}{Tag 0AAW}]{stacks-project}, there is a non-empty open subset $U_{x_1}$ contained in $\PP^n\setminus X\cup h_2(W)\cup h'(W')$. Then it is clear from the construction that for any $p\in U_{x_1}$, the morphism $r_p|_X\colon X\to \PP^{n-1}$ is unramified at $x_1$ and $(r_p|_X)^{-1}(r_p|_X(\{x_1\}))=\{x_1\}$. Moreover, $r_p|_X$ is finite by Lemma \ref{lem-proj-finite}. Then the statement follows from applying Lemma \ref{lem-04DG}. When $k>1$, we take $$U_{x_1,\cdots,x_k}:=\bigcap^k_{i=1} U_{x_i}$$ and the result follows.
\end{proof}

Recall that all schemes in our paper are finite type over $\kk$. For a scheme $X$ and $x\in X$, the \emph{embedding dimension} of $X$ at $x$ is $\dim_{\kappa(x)} \mathrm{T}_x X$ (cf.~\cite[\href{https://stacks.math.columbia.edu/tag/0C2H}{Tag 0C2H}]{stacks-project}). We define the \emph{generic embedding dimension} of $X$ as 
\[\gedim(X):=\min\{\dim_{\kk} \mathrm{T}_x X~|~ x\in U(\kk),~U\text{ is any open dense subset of }X\}.\]
By the upper semi-continuity of the dimension of tangent spaces, $\gedim(X)$ is independent of the choice of the open dense subset $U$. Moreover, there is an open dense subscheme of $X$ such that the embedding dimension at any closed point in it is equal to the generic embedding dimension. Hence, $\gedim$ is a birational invariant.

\begin{remark}\label{rmk-gedim-divisor}
It is clear that $\gedim(X)\geq \dim X$. If $X$ is pure $3$-dimensional and $\dim X_{\mathrm{sing}}\leq 1$, e.g.~$X$ is normal, then any effective divisor $D\subset X$ satisfies $\gedim(D)\leq 3$.
\end{remark}

\begin{remark}\label{rmk-gedim'}
Alternatively, we can define an integer $\gedim'(X)$ for a positive-dimensional scheme $X$ by
\[\gedim'(X):=\min\{\dim_{\kk} \mathrm{T}_x X~|~ x\in U(\kk),~U\text{ is an open subset of }X~\text{such that }\dim X\setminus U=0\}.\]
If $\dim X=1$, then by \cite[Lemma 4.3]{liu-ruan:cast-bound}, we can always find a curve $X'\subset X$ such that $X\setminus X'$ is empty or a finite set and $\gedim(X')=\gedim'(X)$. In particular, if $X$ is proper, then $g(X)\leq g(X')$.
\end{remark}

\begin{lemma}\label{lem-bir}
Let $X\subset \PP^n$ be a closed subscheme with $\dim X<n-1$ and $\gedim(X)<n$. Then there exists an open subset $U\subset \PP^n$ such that for any closed point $p\in U$, $r_p|_X\colon X\to \PP^{n-1}$ is a finite morphism that birational onto its (schematic) image. 
\end{lemma}

\begin{proof}
Let $X_1,\cdots ,X_k$ be irreducible components of $X$. By $\gedim(X)<n$, we can take a closed point $x_i\in X_i$ for each $1\leq i\leq k$ such that $\dim_{\kk} \mathrm{T}_{x_i} X<n$. Then the result follows from defining $U:=U_{x_1,\cdots,x_k}$ and applying Proposition \ref{prop-key-proj}.
\end{proof}

From the discussion above, we can process $n-3$ times projection from a general point and get:

\begin{lemma}\label{lem-good-proj-exist}
Let $X_1,\cdots,X_k\subset \PP^n$ be closed subschemes with $\dim X_i\leq 2$ and $\gedim(X_i)\leq 3<n$ for each $1\leq i\leq k$. Then we can apply $n-3$ times projection from a general point and get a rational map $\pi\colon \PP^n\dashrightarrow \PP^3$ defined on an open subset $U\subset \PP^n$ satisfying

\begin{enumerate}

    \item $\pi|_U$ is smooth and $\pi|_U^*\oh_{\PP^3}(1)\cong \oh_{\PP^n}(1)|_U$,

    \item $X_i\subset U$ for any $1\leq i\leq k$,

    \item $X_i\to \pi(X_i)$ is birational and finite for any $1\leq i\leq k$, and

    \item if $Z\subset \PP^n$ be a closed subscheme such that $Z\neq \PP^n$ and $\dim Z\leq 3$, then we can take $\pi$ that satisfies $Z\subset U$.
\end{enumerate}

\end{lemma}

\begin{proof}
First, (a) follows from the discussion at the beginning of this subsection. From Lemma \ref{lem-bir}, (b) and (c) can be achieved as well. Finally, since $\dim Z\leq 3$, $Z$ can not map onto $\PP^m$ unless $m=3$. Hence from Lemma \ref{lem-proj-finite}, (d) can also be satisfied by projection from general points.
\end{proof}

To simplify the notation, we make the following definition.

\begin{definition}
Let $X\subset \PP^n$ be a closed subscheme with $\dim X\leq 2$ and $\gedim(X)\leq 3<n$. A \emph{good projection of $X$} is a rational map $\pi\colon \PP^n\dashrightarrow \PP^3$ constructed in Lemma \ref{lem-good-proj-exist}. We call the restricted morphism $X\to \pi(X)$ a good projection of $X$ as well and denote it by $\pi_X$.

If $X_1,\cdots,X_k\subset \PP^n$ are closed subschemes with $\dim X_i\leq 2$ and $\gedim(X_i)\leq 3<n$ for each $1\leq i\leq k$, then we say $\pi\colon \PP^n\dashrightarrow \PP^3$ is a \emph{good projection of $\{X_i\}_{1\leq i\leq k}$} if $\pi$ is a good projection of $X_i$ for each $1\leq i\leq k$.
\end{definition}

From the construction in Lemma \ref{lem-good-proj-exist}, a good projection $\pi\colon \PP^n\dashrightarrow \PP^3$ is the composition of $n-3$ rational maps $\pi_i\colon \PP^{n-i+1} \dashrightarrow \PP^{n-i}$ and each of them is a projection from a general point. We call $\pi_i$ \emph{the $i$-th projection of $\pi$}. 

Note that $\pi$ is also a projection from a linear subspace $\PP^{n-4}\subset \PP^n$, so the indeterminacy of $\pi$ is exactly $\PP^{n-4}$. We say a statement holds for \emph{any generic good projection} if this statement holds for any good projection corresponds to a generic linear subspace $\PP^{n-4}\subset \PP^n$.

\begin{remark}\label{rmk-proj-iso}
It is straightforward to check that $\pi\colon \PP^n\dashrightarrow \PP^3$ has a section, in the sense that there is a subspace $i\colon \PP^3\hookrightarrow \PP^n$ such that $\pi\circ i=\mathrm{id}_{\PP^3}$.
\end{remark}

In the rest of this section, we study the behavior of curves under good projections.

\begin{lemma}\label{lem-cok-dim0}
Let $C\subset \PP^n$ be a closed subscheme with $\dim C=1$ and $\gedim(C)\leq 3<n$. Assume $\pi_C\colon C\twoheadrightarrow C'\subset \PP^3$ is a good projection of $C$. Then we have an exact sequence
\begin{equation}\label{eq-key-exact-seq}
    0\to \oh_{C'}\to (\pi_C)_*\oh_C\to T\to 0,
\end{equation}
where $T$ is supported on points, and $T=0$ if and only if $\pi_C$ is an isomorphism. 

In particular, $\deg(C)=\deg(C')$ and
\[g(C)\leq g(C')\leq \frac{(\deg(C)-1)(\deg(C)-2)}{2}.\]
\end{lemma}

\begin{proof}
As $C'$ is the schematic image of $C$, the natural map $\oh_{C'}\to (\pi_C)_*\oh_C$ is injective. Moreover, since $\pi_C$ is birational, we see that $\cok(\oh_{C'}\hookrightarrow (\pi_C)_*\oh_C)$ is zero on an open dense subscheme of $C'$, hence is supported on points, which proves the first statement. Moreover, $T=0$ if and only if $\oh_{C'}\cong (\pi_C)_*\oh_C$, which is equivalent to $C\cong C'$ by the affineness of $\pi_C$.

From the definition of a good projection, we see $\pi_C^*(\oh_{\PP^3}(1)|_{C'})\cong \oh_{\PP^n}(1)|_C$, thus the claim $\deg(C)=\deg(C')$ follows from the birationality of $\pi_C$. Next, using \eqref{eq-key-exact-seq}, we get $\chi(\oh_C)\geq \chi(\oh_{C'})$, which gives $g(C)\leq g(C')$. Finally, applying \cite[Theorem 3.1]{Hartshorne1994TheGO} to $C'\subset \PP^3$ implies
\[g(C)\leq g(C')\leq  \frac{(\deg(C)-1)(\deg(C)-2)}{2}.\]
\end{proof}

\begin{proposition}\label{prop-decompose-wall}
Let $C\subset \PP^n$ be a closed subscheme with $\dim C=1$ and $\gedim(C)\leq 3<n$. Assume $\pi_C\colon C\twoheadrightarrow C'\subset \PP^3$ is a good projection of $C$. If we have an exact sequence
\[0\to \oh_{C_1'}(-k)\to \oh_{C'}\to \oh_{C'\cap S'}\to 0\]
where $S'\in |\oh_{\PP^3}(k)|$ for an integer $k\geq 1$ and $\dim(C'_1)=\dim (C'\cap S')=1$, then there is a closed subscheme $C_1\subset C$ and $S\in |\oh_{\PP^n}(k)|$ satisfying $$g(C)= g(C_1)+g(C\cap S)+k\deg(C_1)-1,$$
with $\deg(C_1)=\deg(C_1')$ and $\deg(C\cap S)=\deg(C'\cap S')$.
\end{proposition}

\begin{proof}
Let $\pi\colon \PP^n\dashrightarrow \PP^3$ be the rational map corresponding to $\pi_C$ which is defined on $U\subset \PP^n$. We define $S:=\overline{\pi_U^{-1}(S')}$. By the flatness of $\pi_U=\pi|_U$, $S$ is a hypersurface in $\PP^n$ and we can assume that $S\in |\oh_{\PP^n}(l)|$ for an integer $l>0$. Moreover, since $\pi_C^*(\oh_{\PP^3}(1)|_{C'})\cong \oh_{\PP^n}(1)|_C$, we see $\pi_C^*(\oh_{\PP^3}(S')|_{C'})\cong \oh_{\PP^n}(S)|_C$. As $S'\in |\oh_{\PP^3}(k)|$, we get $\oh_{\PP^n}(k)|_C\cong \oh_{\PP^n}(S)|_C\cong \oh_{\PP^n}(l)|_C$. Hence by comparing the degree of line bundles on both sides, we find $k=l$.

Next, note that $C\subset U$ and $C\cap \pi_U^{-1}(C'\cap S')=\pi_C^{-1}(C'\cap S')$, we get
\[\pi_C^{-1}(C'\cap S')=C\cap \pi_U^{-1}(C'\cap S')=C\cap \pi_U^{-1}(C')\cap \pi_U^{-1}(S')=C\cap \pi_U^{-1}(S'),\]
where the last equality is from $C\subset \pi_U^{-1}(C')$. Then by $C\subset U$ and $\pi_U^{-1}(S')\cap U=S$, we obtain $$\pi_C^{-1}(C'\cap S')=C\cap \pi_U^{-1}(S')=C\cap S.$$ Therefore, the birationality of $\pi_C$ implies that $C\cap S$ and $C'\cap S'$ only different from a zero-dimensional subscheme, so we can conclude that $\deg(C\cap S)=\deg(C'\cap S')$.

Now from Lemma \ref{lem-decompose}, we have an exact sequence
\[0\to \oh_{C_1}(-k)\to \oh_C\to \oh_{C\cap S}\to 0,\]
where $C_1\subset C$ is a closed subscheme. Then it is clear that $g(C)= g(C_1)+g(C\cap S)+kd_1-1$. Hence it remains to prove $\deg(C_1)=\deg(C_1')$. To this end, since $\pi_C$ is affine, the functor $(\pi_C)_*$ preserves exact sequences. Hence we get a commutative diagram of exact sequences
\[\begin{tikzcd}
	0 & {\oh_{C_1'}(-k)} & {\oh_{C'}} & {\oh_{C'\cap S'}} & 0 \\
	0 & {(\pi_C)_*\oh_{C_1}(-k)} & {(\pi_C)_*\oh_{C}} & {(\pi_C)_*\oh_{C\cap S}} & 0
	\arrow[from=1-1, to=1-2]
	\arrow[from=1-2, to=1-3]
	\arrow[from=1-3, to=1-4]
	\arrow[from=1-4, to=1-5]
	\arrow[from=2-1, to=2-2]
	\arrow[from=2-2, to=2-3]
	\arrow[from=2-3, to=2-4]
	\arrow[from=2-4, to=2-5]
	\arrow["b"', from=1-4, to=2-4]
	\arrow["a"', from=1-3, to=2-3]
	\arrow["c"', from=1-2, to=2-2]
\end{tikzcd}\]
where $a$ and $b$ are injective with zero-dimensional cokernels Lemma \ref{lem-cok-dim0}. Thus, $c$ is injective and $\cok(c)$ is supported on points as well. This gives $\deg(C_1)=\deg(C_1')$ and completes the proof.
\end{proof}






\section{Bounds of the genus} \label{sec_bound}

In this section, we aim to get upper bounds of the arithmetic genus of curves in projective 3-folds and prove Theorem \ref{thm-general-bound}, Theorem \ref{thm-general-bound-isolated}, and their corollaries. We divide the proof into several lemmas and propositions. Let $h$ be the hyperplane class on $\PP^3$.

We start with two lemmas. Define a function $$\epsilon(d, n)\colon \ZZ_{\geq 1}\times \ZZ_{\geq 1}\to \QQ$$ by
\begin{equation}\label{eq-epsilon-def}
\epsilon(d,n):=\frac{f}{2}(n-f-1+\frac{f}{n})=(\frac{1-n}{2n})f^2+\frac{n-1}{2}f,
\end{equation}
where $d\equiv f \mod n$ and $0\leq f<n$. It is easy to see $\epsilon(d,n)=\epsilon(f,n)=\epsilon(n-f,n)$ and
\begin{equation}\label{eq-epsilon-bound}
   0\leq \epsilon(d, n)\leq \frac{n^2-n}{8}. 
\end{equation}

\begin{lemma}\label{lem-ms-bound}
Let $C\subset \PP^3$ be a $1$-dimensional closed subscheme of degree $d$. If $D\in |\oh_{\PP^3}(n)|$ is a divisor such that $C\subset D$ and $I_{C/D}$ is semistable at some $(a,b)\in \mathbb{R}_{>0}\times \mathbb{R}$ (e.g.~$D$ is integral), then we have
\[g(C)\leq \frac{1}{2n}d^2+\frac{n-4}{2}d+1-\epsilon(d, n).\]
\end{lemma}

\begin{proof}
Note that
\[\ch_{h}(I_{C/D})=\big(0,n,-d-\frac{n^2}{2},\frac{n^3}{6}+(g(C)+2d-1)\big),\]
then the result follows from applying \cite[Theorem 3.4]{macri:space-curve} to $I_{C/D}$.
\end{proof}

\begin{lemma}\label{lem-bmt-at-ab}
Let $C\subset \PP^3$ be a $1$-dimensional closed subscheme of degree $d$. If $I_C$ is $\sigma_{a,b_0}$-semistable for any $a>0$ and a fixed $b_0<0$, then we have
\[g(C)\leq \frac{2}{-3b_0}d^2+(-\frac{b_0}{3}-2)d+1.\]
\end{lemma}

\begin{proof}
The result follows from applying Theorem \ref{SBG} to $I_C$ at $(a,b)=(0,b_0)$ directly.
\end{proof}

\subsection{Curves in projective surfaces}

Now we are going to bound the genus of curves in projective surfaces, which will be used later to study curves in projective 3-folds.

We begin with a bound of the genus of curves contained in the surfaces in $\PP^3$.

\begin{proposition}\label{prop-key}
Let $S\in |\oh_{\PP^3}(l)|$ be an effective divisor. We define
\[n:=\min\{D.h^2~|~D~\text{is an effective divisor of } \PP^3 \text{ contained in }S\}.\]
Then there exists an integer $N_{n, l}$ only depends on $n$ and $l$, such that for any $1$-dimensional closed subscheme $C\subset S$ of degree $d\geq N_{n,l}$, we have
\[g(C)\leq \frac{1}{2n}d^2+\frac{n-4}{2}d+1-\epsilon(d, n).\]
\end{proposition}

\begin{proof}
By \cite[Lemma 4.3]{liu-ruan:cast-bound}, to get an upper bound of the genus, we can assume that $C$ is a curve.

Let $N_0\geq 1$ be the smallest integer that satisfies
\begin{equation}\label{eq-N0}
\frac{1}{2s}d^2+\frac{s}{2}d-\epsilon(d, s)\leq \frac{1}{2n}d^2+\frac{n}{2}d-\epsilon(d, n),~\text{for any }d\geq N_0~\text{and integer }n\leq s\leq \frac{4}{3}n.
\end{equation}
Such $N_0$ exists since $0\leq \epsilon(d, s)\leq \frac{s^2-s}{8}$. Note that $N_0$ only depends on $n$.

Let $N_1\geq N_0$ be the smallest integer that satisfies
\begin{equation}\label{eq-b_d}
N_1\geq \max\{\frac{16}{9}n^2, (n-1)l+1\},
\end{equation}
\begin{equation}\label{eq-d>18epsilon}
N_{1}\geq \frac{18}{n}\epsilon(d,n), \forall 0<d\leq n,
\end{equation}
\begin{equation}\label{eq-no-wall}
k\sqrt{2 N_1}-\frac{k^2}{2}\geq kl,~\text{ for any integer }1\leq k\leq n-1,
\end{equation}
\begin{equation}\label{eq-high-deg-bound}
k\sqrt{2 N_1}-\frac{k^2}{2}\geq N_0-1,~\text{ for any integer }n\leq k\leq \frac{4}{3}n,
\end{equation}
and
\begin{equation}\label{eq-jiaocha}
k\sqrt{2 N_1}-\frac{k^2}{2}\geq \frac{4}{3}n^2-1,~\text{ for any integer }n\leq k\leq \frac{4}{3}n.
\end{equation}
Then it is clear from the construction that $N_1$ only depends on $n$ and $l$.

We start with two claims.

\medskip

\textbf{Claim 1.} \emph{If $D\in |\oh_{\PP^3}(k)|$ is an effective divisor such that $k<n$, then $\dim D\cap S=1$ and $\deg D\cap S= kl$.}

Indeed, by our definition of $n$, if $\dim D\cap S=2$, then $(D\cap S).h^2\leq D.h^2=k<n$, which makes a contradiction. Thus $\dim D\cap S=1$ and $\deg D\cap S = kl$. This proves the Claim 1. \qed

\bigskip

\textbf{Claim 2.} \emph{Let $Z\subset S$ be a curve of degree $\deg Z\geq N_{1}$. If $Z$ is not $(-\frac{4}{3}n)$-neutral, then there is a divisor $Y\in |\oh_{\PP^3}(t)|$ with $n\leq t\leq \frac{4}{3}n$ such that either 
\[Z\subset Y~\text{and }g(Z)\leq \frac{1}{2n}(\deg Z)^2+\frac{n-4}{2}\deg Z+1-\epsilon(\deg Z, n),\]
or there is a curve $Z_1\subset Z$ satisfies $\dim Z\cap Y$=1 and}
\begin{enumerate}

    \item $$g(Z\cap Y)\leq \frac{1}{2n}(\deg Z\cap Y)^2+\frac{n-4}{2}\deg Z\cap Y+1-\epsilon(\deg Z\cap Y, n),$$

    \item $$g(Z)=g(Z_1)+g(Z\cap Y)+t \deg(Z_1)-1,$$

    \item $$\deg Z_1<\deg Z+ \frac{t^2}{2}-t\sqrt{2\deg Z},\quad t\sqrt{2\deg Z}-\frac{t^2}{2}<\deg Z\cap Y,$$

    \item $$\frac{4}{3}n-\frac{\deg Z_1\cap Y}{n}\leq 0.$$
\end{enumerate}

\medskip

As $Z$ is not $(-\frac{4}{3}n)$-neutral, we know that there is an uppermost actual wall for $I_Z$ in the range $(a,b)\in \mathbb{R}_{>0}\times (-\frac{4}{3}n, 0)$. Since $b_{\deg(Z)}\leq -\frac{4}{3}n$ by \eqref{eq-b_d}, from Proposition \ref{cor_wall} there is an effective divisor $Y\in |\oh_{\PP^3}(t)|$ with $t\leq \frac{4}{3}n$ such that $Z\subset Y$ or there exists a curve $Z_1\subset Z$ satisfies (b) and (c). In the former case, if $t\leq n-1$, then by Claim 1 we have $$N_1\leq \deg Z\leq \deg Y\cap S=tl\leq (n-1)l,$$ contradicts \eqref{eq-b_d}. Thus $n\leq t\leq \frac{4}{3}n$, and by Lemma \ref{lem-ms-bound} we obtain
\[g(Z)\leq \frac{1}{2t}(\deg Z)^2+\frac{t-4}{2}\deg Z+1-\epsilon(\deg Z, t)\]
since $I_{Z/Y}$ is semistable at points on the uppermost actual wall of $I_Z$. As $\deg Z\geq N_1\geq N_0$, from \eqref{eq-N0} we get $$g(Z)\leq \frac{1}{2t}(\deg Z)^2+\frac{t-4}{2}\deg Z+1-\epsilon(\deg Z, t)\leq \frac{1}{2n}(\deg Z)^2+\frac{n-4}{2}\deg Z+1-\epsilon(\deg Z, n).$$

\bigskip

Now assume that $Z$ is not contained in $Y$. If $t\leq n-1$, then by Claim 1, we get $\dim Y\cap S=1$ and $\deg Y\cap S=tl$. Hence $\deg Z\cap Y\leq \deg S\cap Y=tl$ But this makes a contradiction since $$\deg Z\cap Y>t\sqrt{\deg Z}-\frac{t^2}{2}\geq tl$$ by (c) and \eqref{eq-no-wall}. Hence $n\leq t\leq \frac{4}{3}n$. Then from $\deg Z\geq N_1$, (c) and \eqref{eq-high-deg-bound}, we obtain $\deg Z\cap Y\geq N_0$. Thus Lemma \ref{lem-ms-bound} and \eqref{eq-N0} imply (a) as $I_{Z\cap Y/Y}$ is semistable at the uppermost wall of $I_Z$. Finally, (d) follows from (c), $\deg Z\geq N_1$ and \eqref{eq-jiaocha}. This ends the proof of Claim 2. \qed


\bigskip

Now let $N_{n,l}\geq N_1$ be the smallest integer satisfies
\begin{equation}\label{eq-first-sum}
k+\frac{k^2}{2n}-\frac{k}{n}\sqrt{2N_{n,l}}+\frac{n^2-n}{8}\leq -\frac{N_1^2}{2}~\text{for any integer }n\leq k\leq \frac{4}{3}n.
\end{equation}
Then $N_{n,l}$ only depends on $n$ and $l$ since $N_1$ does.

\bigskip

We divide the proof into finite steps inductively as follows.

\textbf{Step 0.}

Let $C\subset S$ be a curve of degree $d\geq N_{n, l}$. If $C$ is $(-\frac{4}{3}n)$-neutral, then we can apply Lemma \ref{lem-bmt-at-ab} to $(a,b)=(0, -\frac{4}{3}n)$ and get
\[g(C)\leq \frac{1}{2n}d^2+(\frac{4}{9}n-2)d+1.\]
Since $nd\geq nN_{n, l}\geq nN_1 \geq 18\epsilon(d,n)$ by \eqref{eq-d>18epsilon}, we see 
\[g(C)\leq \frac{1}{2n}d^2+(\frac{4}{9}n-2)d+1\leq \frac{1}{2n}d^2+\frac{n-4}{2}d+1-\epsilon(d, n),\]
which gives the desired bound.

Now assume that $C$ is not $(-\frac{4}{3}n)$-neutral. If the uppermost wall is given by a divisor containing $C$, i.e.~it is of form $W(\oh_{\PP^3}(-k_1), I_C)$, then by the first case of Claim 2,
\[g(C)\leq \frac{1}{2n}d^2+\frac{n-4}{2}d+1-\epsilon(d, n)\]
and we are done. If the uppermost wall is given by an effective divisor $D_1\in |\oh_{\PP^3}(k_1)|$ and a curve $C_1\subset C$ of degree $d_1$, i.e.~it is of form $W(I_{C_1}(-D_1), I_C)$, then by letting $Z:=C$, $Y:=D_1$ and $Z_1:=C_1$ in Claim 2, all properties (a)--(d) are satisfied and $n\leq k_1\leq \frac{4}{3}n$. We denote $d_0':=\deg (C\cap D_1) =d-d_1$.

Here are three possible situations of the curve $C_1$:

\begin{itemize}
    \item $d_1< N_1$,

    \item $d_1\geq N_1$, and $C_1$ is $(-\frac{4}{3}n)$-neutral or the uppermost wall is given by a line bundle,

    \item $d_1\geq N_1$, $C_1$ is not $(-\frac{4}{3}n)$-neutral, and the uppermost wall is not given by line bundles.
\end{itemize}
If we are in the first two cases, then we end our process at Step 0. If we are in the third situation, then we go to the following Step 1.

\medskip

\textbf{Step 1.} \emph{$d_1\geq N_1$, $C_1$ is not $(-\frac{4}{3}n)$-neutral, and the uppermost wall is not given by line bundles.}

Using Claim 2 again, we can find a curve $C_2\subset C_1$ of degree $d_2$ and $D_2\in |\oh_{\PP^3}(k_2)|$ satisfying the conditions in Claim 2, where $Z:=C_1, Y:=D_2$ and $Z_1:=C_2$ with $n\leq k_2\leq \frac{4}{3}n$. We denote $d_1':=\deg(C_1\cap D_2)=d_1-d_2$. Then we can run the same process for $C_2$ and we also have the following three possibilities:
\begin{itemize}
    \item $d_2< N_1$,

    \item $d_2\geq N_1$, and $C_2$ is $(-\frac{4}{3}n)$-neutral or the uppermost wall is given by a line bundle,

    \item $d_2\geq N_1$, $C_2$ is not $(-\frac{4}{3}n)$-neutral, and the uppermost wall is not given by line bundles.
\end{itemize}
If the first two cases happen, then we end our proof at Step 1. If the third situation happens, then by Claim 2 again, we can find a curve $C_3\subset C_2$ of degree $d_3$ and $D_2\in |\oh_{\PP^3}(k_2)|$ satisfy the conditions in Claim 2. Then we go to the next Step 2. 

\medskip

Inductively, for any integer $m\geq 1$, we have:

\textbf{Step m-1.} \emph{$d_{m-1}\geq N_1$, $C_{m-1}$ is not $(-\frac{4}{3}n)$-neutral, and the uppermost wall is not given by line bundles.}

By Claim 2, we can find a curve $C_{m}\subset C_{m-1}$ of degree $d_m$ and $D_m\in |\oh_{\PP^3}(k_m)|$ satisfying the conditions in Claim 2, where $Z:=C_{m-1}, Y:=D_m$ and $Z_1:=C_m$ with $n\leq k_m\leq \frac{4}{3}n$. We denote $d_{m-1}':=\deg(C_{m-1}\cap D_m)=d_{m-1}-d_m$. Then we can run the same process for $C_m$ and we have the following three possibilities:
\begin{itemize}
    \item $d_m< N_1$,

    \item $d_m\geq N_1$, and $C_m$ is $(-\frac{4}{3}n)$-neutral or the uppermost wall is given by a line bundle,

    \item $d_m\geq N_1$, $C_m$ is not $(-\frac{4}{3}n)$-neutral, and the uppermost wall is not given by line bundles.
\end{itemize}
We end at this step if the first two situations happen and go to the next Step m otherwise.

\bigskip

Note that $d>d_1>d_2>d_3>\dots$, so if we continue this process, we will end at Step m-1 for an integer $m\geq 1$, i.e.~the first two situations for $C_m$ above happen. Therefore, we finally get a sequence of curves $C_m\subset C_{m-1}\subset \cdots C_2\subset C_1\subset C_0:=C$ for an integer $m\geq 1$ and divisors $D_i\in |\oh_{\PP^3}(k_i)|$ with $n\leq k_i \leq \frac{4}{3}n$ for any $1\leq i\leq m$ such that

\begin{enumerate}[(i)]
    \item $d_m<N_1$, or $d_m\geq N_1$ and $C_m$ is $(-\frac{4}{3}n)$-neutral or the uppermost wall is given by a line bundle.

    \item $C_i$ and $D_i$ satisfy the conditions in Claim 2 for any $1\leq i\leq m$, where $Z:=C_{i-1}, Y:=D_i$ and $Z_1:=C_i$.
\end{enumerate}
In particular, by Claim 2 we have
\begin{equation}\label{eq-gCi}
g(C_{i-1})=g(C_i)+g(C_{i-1}\cap D_i)+k_id_i-1
\end{equation}
for any $1\leq i\leq m$, where $C_0:=C$. Following the notations above, we denote $$d_i':=d_i-d_{i+1}=\deg(C_i\cap D_{i+1})$$ for any $0\leq i\leq m-1$, where $d_0:=d$. By Claim 2, for any $0\leq i\leq m-1$ we have
\begin{equation}\label{eq-112}
    g(C_i\cap D_{i+1})\leq \frac{1}{2n}d_i'^2+\frac{n-4}{2}d_i'+1-\epsilon(d_i', n).
\end{equation}
Now sum up \eqref{eq-gCi} over $i$ gives
\[g(C)=\sum_{i=0}^{m-1} g(C_{i}\cap D_{i+1})+g(C_m)+\sum_{j=1}^m k_jd_j-m\]
\[\overset{\eqref{eq-112}}{\leq} \big(\sum^{m-1}_{i=0} \frac{1}{2n}d_i'^2+\frac{n-4}{2}d_i'+1-\epsilon(d_i', n)\big)+g(C_m)+\sum_{j=1}^m k_jd_j-m=\]
\begin{equation}\label{eq-11}
    \big( \frac{1}{2n}(\sum^{m-1}_{i=0} d_i')^2-\frac{\sum_{0\leq x,y\leq m-1, x\neq y}d_x'd_y'}{n}+\frac{n-4}{2}(\sum^{m-1}_{i=0} d_i')-\sum^{m-1}_{i=0}\epsilon(d_i', n)\big)+g(C_m)+\sum_{j=1}^m k_jd_j.
\end{equation}
Note that 
\begin{equation}\label{eq-sum-dj}
d_j=\sum_{i=j}^{m-1} d_i'+d_m
\end{equation}
for any $0\leq j\leq m-1$, we get
\[\eqref{eq-11}=\frac{1}{2n}(d-d_m)^2+\frac{n-4}{2}(d-d_m)+g(C_m)+\sum_{j=1}^m k_jd_j-\frac{\sum_{0\leq x,y\leq m-1, x\neq y}d_x'd_y'}{n}-\sum^{m-1}_{i=0}\epsilon(d_i', n)\]
\[=\frac{1}{2n}d^2+\frac{n-4}{2}d+1-\epsilon(d,n)\]
\[+\big(-\frac{d d_m}{n}+\frac{d_m^2}{2n}-\frac{n-4}{2}d_m-1+\epsilon(d,n)+ g(C_m)+\sum_{j=1}^m k_jd_j-\frac{\sum_{0\leq x,y\leq m-1, x\neq y}d_x'd_y'}{n}-\sum^{m-1}_{i=0}\epsilon(d_i', n) \big)\]
\[\overset{\eqref{eq-epsilon-bound}}{\leq} \frac{1}{2n}d^2+\frac{n-4}{2}d+1-\epsilon(d,n)\]
\[+\big(-\frac{d d_m}{n}+\frac{d_m^2}{2n}-\frac{n-4}{2}d_m-1+\epsilon(d,n)+ g(C_m)+\sum_{j=1}^m k_jd_j-\frac{\sum_{0\leq x,y\leq m-1, x\neq y}d_x'd_y'}{n} \big)\]
Thus we only need to prove
\[R_m:=-\frac{dd_m}{n}+\frac{d_m^2}{2n}-\frac{n-4}{2}d_m-1+\epsilon(d,n)+ g(C_m)+\sum_{j=1}^m k_jd_j-\frac{\sum_{0\leq x,y\leq m-1, x\neq y}d_x'd_y'}{n}\leq 0.\]
If $d_m\geq N_1$, and $C_m$ is $(-\frac{4}{3}n)$-neutral or the uppermost wall is given by a line bundle, from Claim 2 we obtain
\[g(C_m)\leq \frac{1}{2n}d_m^2+\frac{n-4}{2}d_m+1-\epsilon(d_m, n).\]
In this case,
\[R_m\leq -\frac{dd_m}{n}+\frac{d_m^2}{n}+\epsilon(d,n)+\sum_{j=1}^m k_jd_j-\frac{\sum_{0\leq x,y\leq m-1, x\neq y}d_x'd_y'}{n}\]
\[\overset{\eqref{eq-sum-dj}}{=}\frac{-d_m(d_m+\sum_{i=0}^{m-1}d_i')}{n}+\frac{d_m^2}{n}+\epsilon(d,n)+\sum_{j=1}^m k_jd_j-\frac{\sum_{0\leq x,y\leq m-1, x\neq y}d_x'd_y'}{n}\]
\[=\frac{-d_m\sum_{i=0}^{m-1}d_i'}{n}+\epsilon(d,n)+\sum_{j=1}^m k_jd_j-\frac{\sum_{0\leq x,y\leq m-1, x\neq y}d_x'd_y'}{n}\]
\[=-\frac{\sum^{m-1}_{i=0}d_i'(d_{i+1}'+\dots +d'_{m-1}+d_m)}{n}+\epsilon(d,n)+\sum_{j=1}^m k_jd_j\]
\[\overset{\eqref{eq-sum-dj}}{=}-\frac{\sum_{i=0}^{m-1} d_{i}'d_{i+1}}{n}+\sum_{j=1}^m k_jd_j+\epsilon(d,n)\]
\[=\sum_{j=1}^m d_j(k_j-\frac{d_{j-1}'}{n})+\epsilon(d,n)\]

Since $d_{1}>\cdots >d_{m-1}\geq N_1$, by \eqref{eq-jiaocha} and (c) of Claim 2, we have 
\begin{equation}\label{eq-111}
    d_{j-1}'>k_{j}\sqrt{2 d_{j-1}}-\frac{k_{j}^2}{2}\geq \frac{4}{3}n^2-1\geq nk_{j}-1
\end{equation}
for any $1\leq j\leq m$. Thus
\[\sum_{j=1}^m d_j(k_j-\frac{d'_{j-1}}{n})+\epsilon(d,n)\leq d_1(k_1-\frac{d_0'}{n})+\epsilon(d,n)\overset{\eqref{eq-epsilon-bound}}{\leq} d_1(k_1-\frac{d_0'}{n})+\frac{n^2-n}{8}.\]
Then $R_m\leq 0$ follows from \eqref{eq-first-sum} since
\[d_1(k_1-\frac{d_0'}{n})+\frac{n^2-n}{8}\leq k_1-\frac{d_0'}{n}+\frac{n^2-n}{8}<k_1+\frac{k_1^2}{2n}-\frac{k_1}{n}\sqrt{2d}+\frac{n^2-n}{8}\leq 0.\]

If $d_m<N_1$, then we have
\[R_m=-\frac{dd_m}{n}+\frac{d_m^2}{2n}-\frac{n-4}{2}d_m-1+\epsilon(d,n)+ g(C_m)+\sum_{j=1}^m k_jd_j-\frac{\sum_{0\leq x,y\leq m-1, x\neq y}d_x'd_y'}{n}\]
\[\overset{\eqref{eq-sum-dj}}{=}\frac{-d_m\sum^{m-1}_{i=0}d_i'}{n}-\frac{d_m^2}{2n}-\frac{n-4}{2}d_m-1+\epsilon(d,n)+ g(C_m)+\sum_{j=1}^m k_jd_j-\frac{\sum_{0\leq x,y\leq m-1, x\neq y}d_x'd_y'}{n}\]
\[=\sum_{j=1}^m d_j(k_j-\frac{d_{j-1}'}{n})-\frac{d_m^2}{2n}-\frac{n-4}{2}d_m-1+\epsilon(d,n)+ g(C_m)\]
\[\overset{\eqref{eq-111}}{\leq} k_1-\frac{d_0'}{n}-\frac{d_m^2}{2n}-\frac{n-4}{2}d_m-1+\epsilon(d,n)+ g(C_m)\]
\[\overset{\eqref{eq-epsilon-bound}\text{+}\ref{lem-cok-dim0}}{\leq} k_1-\frac{d_0'}{n}-\frac{d_m^2}{2n}-\frac{n-4}{2}d_m-1+\frac{n^2-n}{8}+\frac{(d_m-1)(d_m-2)}{2}\]
\[\leq k_1-\frac{d_0'}{n}+\frac{n^2-n}{8}+\frac{d_m^2}{2}+\frac{1-n}{2}d_m\]
\[\overset{\mathrm{d_m<N_1}}{\leq} k_1-\frac{d_0'}{n}+\frac{n^2-n}{8}+\frac{N_1^2}{2}\]
\[<k_1+\frac{k_1^2}{2n}-\frac{k_1}{n}\sqrt{2d}+\frac{n^2-n}{8}+\frac{N_1^2}{2}\leq 0,\]
where the last inequality follows from \eqref{eq-first-sum}.

Thus in both cases, we have $R_m\leq 0$ and we can conclude that
\[g(C)\leq \frac{1}{2n}d^2+\frac{n-4}{2}d+1-\epsilon(d, n)+R_m\leq \frac{1}{2n}d^2+\frac{n-4}{2}d+1-\epsilon(d, n).\]
\end{proof}


As a corollary, we can obtain a bound for curves of sufficiently large degree in projective surfaces.

\begin{corollary}\label{cor-projective-surface}
Let $S$ be a pure $2$-dimensional projective scheme with $\gedim(S) \leq 3$ and $H$ be an ample divisor on it. Fix an integer $m_H$ such that $m_HH$ is very ample. We define an integer $n_H^S$ as
\[n_H^S:=\min\{Y.(m_H H)^2~|~Y~\text{is a pure two-dimensional closed subscheme contained in }S\}.\]
Then there exists an integer $N_H^S$ such that for any $1$-dimensional closed subscheme $C\subset S$ of $H$-degree $d\geq N^S_H$ and $\gedim'(C)\leq 3$, we have
\[g(C)\leq \frac{1}{2n_H^S}(m_H d)^2+\frac{n_H^S-4}{2}(m_H d)+1-\epsilon(m_H d, n_H^S).\]
\end{corollary}

\begin{proof}
Since $\gedim'(C)\leq 3$, by Remark \ref{rmk-gedim'} we can assume that $C$ is a curve and $\gedim(C)\leq 3$ since we only need to get an upper bound of $g(C)$. 

Let $i\colon S\hookrightarrow \PP^m=\PP(|m_H H|)$ be the embedding induced by $|m_H H|$. As $\gedim(S)\leq 3$, we can apply Lemma \ref{lem-good-proj-exist} and find a good projection $\pi\colon \PP^m\dashrightarrow \PP^3$ of both $C$ and $S$ such that $\pi_S$ is regular. By definition, we have $\pi_S^*\oh_{\PP^3}(h)=\oh_S(m_H H)$. Moreover, since $\pi_C$ and $\pi_S$ are both birational onto their images $C'$ and $S'$, respectively, we have 
\[m_H d=C.(m_H H)=C.\pi_S^*h=(\pi_{S*}C). h=C'.h\]
where $\pi_S^*$ and $\pi_{S*}$ are Chow-theoretic pullback and pushforward. And since $\pi_S$ is birational onto $D'$, we see
\[n_H^S=\min\{Y.(m_H H)^2~|~Y~\text{is a pure two-dimensional closed subscheme contained in }S\}\]
\[=\min\{Y'.h^2~|~Y'~\text{is a pure two-dimensional closed subscheme contained in }S'\}.\]
Note that $S'.h^2=(h|_S)^2=m^2_HH^2$, then applying Proposition \ref{prop-key}, we can define an integer $N_H^S:=\frac{1}{m_H}N_{n^S_H, m_H^2H^2}$ such that when $\frac{1}{m_H}\deg_h(C')=d\geq N_H^S$, we have
\[g(C')\leq \frac{1}{2n_H^S}(m_H d)^2+\frac{n_H^S-4}{2}(m_H d)+1-\epsilon(m_H d, n_H^S)\]
and the result follows from $g(C)\leq g(C')$ by Lemma \ref{lem-cok-dim0}.
\end{proof}

\begin{remark}
The assumption $\gedim(S)\leq 3$ is mild, as it holds for any generically reduced $S$. When $S$ is integral and $H$ above is already very ample, i.e.~$m_H=1$, we get $n_H^S=H^2$. In this case, for any $1$-dimensional closed subscheme $C\subset S$ of $H$-degree $d\geq N^S_H$, we have
\[g(C)\leq \frac{1}{2H^2}d^2+\frac{H^2-4}{2}d+1-\epsilon(d,H^2).\]
\end{remark}

\subsection{Curves in projective 3-folds}

Next, we use the results above to give a bound for the genus of curves in projective 3-folds. We start with a direct corollary of Corollary \ref{cor-projective-surface}, which bounds the genus of curves in a divisor of a projective 3-fold.

\begin{corollary}\label{cor-in-divisor}
Let $(X, H)$ be a polarised projective 3-fold with $\dim X_{\mathrm{sing}} \leq 1$ and $S$ be an effective Cartier divisor on $X$. Fix an integer $m_H$ such that $m_HH$ is very ample. We define an integer $n_H^{|S|}$ as
\begin{equation}\label{eq-nHS}
n_H^{|S|}:=\min\{D.(m_H H)^2~|~D~\text{is an effective divisor of } X \text{ contained in a divisor in }|S|.\}
\end{equation}
Then there exists an integer $N_{H}^{|S|}$ such that for any $1$-dimensional closed subscheme $C\subset X$ of $H$-degree $d\geq N_{H}^{|S|}$ and $\gedim'(C)\leq 3$ containing in an effective divisor in $|S|$, we have
\[g(C)\leq \frac{1}{2n_H^{|S|}}(m_Hd)^2+\frac{n_H^{|S|}-4}{2}(m_Hd)+1-\epsilon(m_H d,n_H^{|S|}).\]
\end{corollary}

\begin{proof}
Assume that $C\subset D$, where $D\in |S|$. By $\dim X_{\mathrm{sing}}\leq 1$ and Remark \ref{rmk-gedim-divisor}, we have $\gedim(D)\leq 3$. Then from Corollary \ref{cor-projective-surface}, we can find an integer $N_H^D$ such that when $d\geq N_H^D$, we have
\[g(C)\leq \frac{1}{2n_H^D}(m_H d)^2+\frac{n_H^D-4}{2}(m_H d)+1-\epsilon(m_H d, n_H^D).\]

By definition, we see $n^{|S|}_H\leq n_H^D\leq S.(m_H H)^2$. Therefore, we can find an integer $N_H^{|S|}\geq N_H^D$ such that when $d\geq N_H^{|S|}$, we have
\[\frac{1}{2n_H^D}(m_H d)^2+\frac{n_H^D-4}{2}(m_H d)+1-\epsilon(m_H d, n_H^D) \leq \frac{1}{2n_H^{|S|}}(m_Hd)^2+\frac{n_H^{|S|}-4}{2}(m_Hd)+1-\epsilon(m_H d,n_H^{|S|})\]
and the result follows.
\end{proof}

\begin{definition}\label{def-nH}
Let $(X, H)$ be a polarised projective 3-fold and $m_H$ be an integer such that $m_H H$ is very ample. We define an integer $n_H$ by
\[n_H:=\min\{D.(m_H H)^2~|~D~\text{is a prime divisor contained in a divisor in }|k m_H H|,~\forall  k\in [1, \frac{4 (m_H H)^3}{3}]\}\]
\[=\min\{D.(m_H H)^2~|~D~\text{is an effective divisor contained in a divisor in }|k m_H H|,~\forall  k\in [1, \frac{4 (m_H H)^3}{3}]\}.\]
\end{definition}

\begin{remark}\label{rmk-nH}
The equality in the above definition is obvious since every effective divisor is a positive linear combination of prime divisors. 

As $H$ is ample, it is clear that $n_H\geq 1$. By definition, we have $n_H\leq m_H^3 H^3$ since $|m_H H|$ is non-empty. More generally, if $|sH|\neq \varnothing$ for a positive integer $s$, then by definition, we have $n_H\leq sH.(m_H H)^2=sm^2_H H^3$. Furthermore, if $X$ is factorial and $\NS(X)=\ZZ H$ with $s$ be the least integer such that $|sH|\neq \varnothing$, then it is clear that $sm^2_H H^3\leq n_H$, which implies $n_H=sm^2_H H^3$ in this case.
\end{remark}

Now we are ready to state and prove one of our main theorems for projective 3-folds:

\begin{theorem}\label{thm-general-bound}
Let $(X, H)$ be a polarised projective 3-fold with $\dim X_{\mathrm{sing}}\leq 1$ and $m_H$ be an integer such that $m_H H$ very ample. Then there exists an integer $N_{H}$ such that for any $1$-dimensional closed subscheme $C\subset X$ of $H$-degree $d\geq N_{H}$ and $\gedim'(C)\leq 3$, we have
\[g(C)\leq \frac{1}{2n_H}(m_H d)^2+\frac{n_H-4}{2}(m_H d)+1-\epsilon(m_H d, n_H).\]
\end{theorem}

\begin{proof}
In this proof, we denote by $n:=n_H$ for simplicity. Since $\gedim'(C)\leq 3$, by Remark \ref{rmk-gedim'} we can assume that $C$ is a curve and $\gedim(C)\leq 3$, since we aim to get an upper bound of $g(C)$. Let $N_2$ be the integer defined by
\begin{equation}\label{eq-N0-thm}
N_2:=\max\{N_H^{|km_H H|}\}_{1\leq k \leq \frac{4}{3}n},
\end{equation}
where $N_H^{|km_H H|}$ is defined in Corollary \ref{cor-in-divisor}.

Let $N_3\geq N_2$ be the smallest integer that satisfies
\begin{equation}\label{eq-N3}
\frac{1}{2s}(m_H d)^2+\frac{s}{2}(m_Hd)-\epsilon(m_H d, s)\leq \frac{1}{2n}(m_H d)^2+\frac{n}{2}(m_H d)-\epsilon(m_H d, n)
\end{equation}
for any $d\geq N_3$ and integer $n\leq s \leq \max\{n_H^{|km_H H|}\}_{1\leq k \leq \frac{4}{3}n}$, where $n_H^{|km_H H|}$ is defined in \eqref{eq-nHS}. Note that by definition, we have $n=\min\{n_H^{|km_H H|}\}_{1\leq k \leq \frac{4}{3}(m_H H)^3}\leq \min\{n_H^{|km_H H|}\}_{1\leq k \leq \frac{4}{3}n}$.

We start with two lemmas.

\bigskip

\textbf{Claim 1.} \emph{Let $C\subset X$ be a $1$-dimensional closed subscheme of $H$-degree $d\geq N_3$ satisfying $\gedim'(C)\leq 3$. If $C$ is contained in an effective divisor in $|km_H H|$ for $1\leq k\leq \frac{4}{3}n$, we have
\[g(C)\leq \frac{m_H^2}{2n}d^2+\frac{n m_H-4m_H}{2}d+1-\epsilon(m_H d, n).\]
}

Since $C$ is contained in an effective Cartier divisor in $|km_H H|$ for $d\geq N_2\geq N_H^{|km_H H|}$, by Corollary \ref{cor-in-divisor} we have
\[g(C)\leq \frac{m_H^2}{2n_H^{|km_H H|}}d^2+\frac{n_H^{|km_H H|} m_H-4m_H}{2}d+1-\epsilon(m_H d,n_H^{|km_H H|}).\]
Then the claim follows from $d\geq N_3$ and \eqref{eq-N3}. \qed

\bigskip

Now let $N_4\geq N_3$ be the smallest integer that satisfies
\begin{equation}\label{eq-b_d-thm}
N_4\geq \frac{16}{9}n^2,
\end{equation}
\begin{equation}\label{eq-d>18epsilon-thm}
N_4 \geq \frac{18}{n}\epsilon(m_H d,n) \text{ for any integer } 0<d\leq n,
\end{equation}
\begin{equation}\label{eq-high-deg-bound-thm}
k\sqrt{2 m_H N_4}-\frac{k^2}{2}\geq m_H(N_3-1)\text{ for any integer }1\leq k\leq \frac{4}{3}n,
\end{equation}
and
\begin{equation}\label{eq-jiaocha-thm}
k\sqrt{2 m_H N_4}-\frac{k^2}{2}\geq kn-1\text{ for any integer }1\leq k\leq \frac{4}{3}n.
\end{equation}

Finally, let $N_H\geq N_4$ be the smallest integer that satisfies
\begin{equation}\label{eq-sum-thm}
k+\frac{k^2}{2n}-\frac{k}{n}\sqrt{2m_H N_H}+\frac{n^2-n}{8m_H}\leq -\frac{m_H N_4^2}{2}~\text{for any integer }1\leq k\leq \frac{4}{3}n.
\end{equation}

\bigskip

As in the proof of Proposition \ref{prop-key}, we divide our proof into finite steps as follows. 

\textbf{Step 0.}

Let $C_0:=C\subset X$ be a $1$-dimensional closed subscheme of $H$-degree $d\geq N_H$ and $\gedim'(C)\leq 3$. Without loss of generality, we can assume that $C$ is a curve and $\gedim(C)\leq 3$. We denote by $\iota \colon X\hookrightarrow \PP(|m_H H|)$ the embedding induced by $|m_H H|$. According to Lemma \ref{lem-good-proj-exist}, we can find a good projection $\pi\colon \PP(|m_H H|)\dashrightarrow \PP^3$ of $C$ such that $\pi_X$ is regular. By definition, we have $\pi_X^*\oh_{\PP^3}(h)=\oh_X(m_H H)$. Since $\pi_C$ is birational onto its image $C'$, we have 
\[m_H d=C.(m_H H)=C.\pi_X^*h=(\pi_{X*}C). h=C'.h=\deg_h C'\]
where $\pi_X^*$ and $\pi_{X*}$ are Chow-theoretic pullback and pushforward. Here are only three possible situations for $C_0':=C'$:

\begin{itemize}
    \item Type I: $C'_0$ is $(-\frac{4}{3}n)$-neutral.

    \item Type II: $C_0'$ is contained in an effective divisor in $|\oh_{\PP^3}(t)|$, where $1\leq t\leq \frac{4}{3}n$.

    \item Type III: $C'_0$ is not contained in any effective divisor in $|\oh_{\PP^3}(t)|$ for any $1\leq t\leq \frac{4}{3}n$, and it is not $(-\frac{4}{3}n)$-neutral.
\end{itemize}

If $C'_0$ is of Type I or Type II, we end the proof at this step.

If $C'_0$ is of Type III, as $d\geq N_1$, by \eqref{eq-b_d-thm} and Proposition \ref{cor_wall} the uppermost actual wall is of form $W(I_{C'_1}(-D'_1), I_{C'_0})$, where $D'_1\in |\oh_{\PP^3}(k_1)|$ for $1\leq k_1\leq \frac{4}{3}n$ and $C'_1\subset C'_0$ is a $1$-dimensional closed subscheme of $h$-degree $\deg_h C'_1=m_H d_1$. Moreover, applying Proposition \ref{prop-decompose-wall}, we can find a closed subscheme $C_1\subset C_0$ and $D_1\in |k_1m_H H|$ such that $\deg_H(C_1)=d_1$, $\deg_H(C\cap D_1)=d_0':=d-d_1$ and
\[g(C_0)=g(C_1)+g(C\cap D_1)+k_1m_Hd_1-1.\]

Note that $\pi$ is also a good projection of $C_1$. If $d_1<N_4$, or $d_1\geq N_4$ and $C_1'$ is of Type I or II, then we end our proof at this step as well. Otherwise, we replace the role of $C_0$ with $C_1$ and go to the following Step 1.

\bigskip

\textbf{Step 1.} \emph{$d_1\geq N_4$ and $C_1'$ is of Type III.}

In this case, by \eqref{eq-b_d-thm} and Proposition \ref{cor_wall}, the uppermost actual wall for $I_{C_1'}$ is of the form $W(I_{C'_2}(-D'_2), I_{C'_1})$, where $D'_2\in |\oh_{\PP^3}(k_2)|$ for $1\leq k_2\leq \frac{4}{3}n$ and $C'_2\subset C'_1$ is a $1$-dimensional closed subscheme of $h$-degree $\deg_h C'_2=m_H d_2$. Moreover, applying Proposition \ref{prop-decompose-wall} as the previous step, we can find a closed subscheme $C_2\subset C_1$ and $D_2\in |k_2m_H H|$ such that $\deg_H(C_2)=d_2$, $\deg_H(C_1\cap D_2)=d_1':=d_1-d_2$ and
\[g(C_1)=g(C_2)+g(C_1\cap D_2)+k_2m_Hd_2-1.\]

Note that $\pi$ is also a good projection of $C_2$. If $d_2<N_4$, or $d_2\geq N_4$ and $C_2'$ is of Type I or II, then we end our process at this step. Otherwise, we replace the role of $C_1$ with $C_2$ and go to the next Step 2.


\bigskip

Inductively, for any integer $m\geq 1$, we have:

\textbf{Step m-1.} \emph{$d_{m-1}\geq N_4$ and $C_{m-1}'$ is of Type III.}

In this case, the uppermost actual wall for $I_{C_{m-1}'}$ is of the form $W(I_{C'_m}(-D'_m), I_{C'_{m-1}})$, where $D'_m\in |\oh_{\PP^3}(k_m)|$ for $1\leq k_m\leq \frac{4}{3}n$ and $C'_m\subset C'_{m-1}$ is a $1$-dimensional closed subscheme of $h$-degree $\deg_h C'_m=m_H d_m$. Moreover, applying Proposition \ref{prop-decompose-wall} again as in the previous step, we can find a closed subscheme $C_m\subset C_{m-1}$ and $D_m\in |k_m m_H H|$ such that $\deg_H(C_m)=d_m$, $\deg_H(C_{m-1}\cap D_m)=d_{m-1}':=d_{m-1}-d_m$ and
\[g(C_{m-1})=g(C_m)+g(C_{m-1}\cap D_m)+k_m m_Hd_m-1.\]

Note that $\pi$ is also a good projection of $C_m$. If $d_m<N_4$, or $d_m\geq N_4$ and $C_m'$ is of Type I or II, then we end our process as well. Otherwise, we replace the role of $C_{m-1}$ with $C_m$ and go to the next Step m.

\bigskip

Since $d_0:=d>d_1>d_2>d_3>\dots$, so if we continue this process, we will end at Step m-1 for an integer $m\geq 1$. Therefore, we finally get sequences of curves $$C_m\subset C_{m-1}\subset \dots C_2\subset C_1\subset C_0:=C$$ and $$C'_m\subset C'_{m-1}\subset \dots C'_2\subset C'_1\subset C'_0:=C'$$ and divisors $D'_i\in |\oh_{\PP^3}(k_i)|$ and $D_i\in |\oh_{X}(k_i m_H H)|$ for any $1\leq i\leq m$ such that $$\deg_h C'_i=m_H \deg_H C_i=m_H d_i$$ satisfying

\begin{enumerate}[(i)]
    \item $d_m<N_4$, or $d_m\geq N_4$ and $C_m'$ is of Type I or II,

    \item $g(C_{i-1})=g(C_i)+g(C_{i-1}\cap D_i)+k_im_Hd_i-1$ for any $1\leq i\leq m$,

    \item $d_i\geq N_4$ for each $0\leq i\leq m-1$,

    \item $1\leq k_i\leq \frac{4}{3}n$ for any $1\leq i\leq m$, and

    \item $m_H d_i'>k_{i+1}\sqrt{2m_H d_i}-\frac{k_{i+1}^2}{2}$ for any $0\leq i\leq m-1$.
\end{enumerate}
Here, (i) is the condition that we stop our process, and (v) follows from Proposition \ref{cor_wall}(a). By (iii), (v), and \eqref{eq-high-deg-bound-thm}, we see
\[\deg_h(C_i'\cap D'_{i+1})=m_H d_i'>k_{i+1}\sqrt{2m_Hd_i}-\frac{k_{i+1}^2}{2}\geq m_H(N_3-1),\]
which implies $d_i'\geq N_3$ for any $0\leq i\leq m$. Thus from Claim 1, we have
\begin{equation}\label{eq-gcidi+1}
    g(C_{i}\cap D_{i+1})\leq \frac{m_H^2}{2n}d_i'^2+\frac{n m_H-4m_H}{2}d_i'+1-\epsilon(m_H d_i', n)
\end{equation}
for any $0\leq i\leq m-1$. Note that 
\begin{equation}\label{eq-dj'}
    d_j=\sum_{i=j}^{m-1} d_i'+d_m
\end{equation}
for any $0\leq j\leq m-1$. Then as in the proof of Proposition \ref{prop-key}, we have
\[g(C)=\sum_{i=0}^{m-1} g(C_{i}\cap D_{i+1})+g(C_m)+\sum_{j=1}^m k_jm_Hd_j-m\]
\[\overset{\eqref{eq-gcidi+1}}{\leq} \big(\sum^{m-1}_{i=0} \frac{1}{2n}(m_Hd_i')^2+\frac{n-4}{2}(m_Hd_i')+1-\epsilon(m_Hd_i', n)\big)+g(C_m)+\sum_{j=1}^m k_jm_Hd_j-m\]
\[= \big( \frac{1}{2n}(\sum^{m-1}_{i=0} m_Hd_i')^2-\frac{\sum_{0\leq x,y\leq m-1, x\neq y}(m_Hd_x')(m_Hd_y')}{n}+\frac{n-4}{2}(m_H\sum^{m-1}_{i=0} d_i')-\sum^{m-1}_{i=0}\epsilon(m_Hd_i', n)\big)\]
\[+g(C_m)+\sum_{j=1}^m k_jm_Hd_j\]
\[\overset{\eqref{eq-dj'}}{=}\frac{m_H^2}{2n}(d-d_m)^2+\frac{nm_H-4m_H}{2}(d-d_m)+g(C_m)+\sum_{j=1}^m k_jm_Hd_j-\frac{m_H^2\sum_{0\leq x,y\leq m-1, x\neq y}d_x'd_y'}{n}-\sum^{m-1}_{i=0}\epsilon(m_Hd_i', n)\]
\[=\frac{m_H^2}{2n}d^2+\frac{nm_H-4m_H}{2}d+1-\epsilon(m_Hd,n)\]
\[+\big(-\frac{m_H^2 d d_m}{n}+\frac{m_H^2d_m^2}{2n}-\frac{m_Hn-4m_H}{2}d_m+\sum_{j=1}^m k_jm_Hd_j-\frac{m_H^2\sum_{0\leq x,y\leq m-1, x\neq y}d_x'd_y'}{n}-\sum^{m-1}_{i=0}\epsilon(m_Hd_i', n) \big)\]
\[-1+\epsilon(m_Hd,n)+ g(C_m)\]
\[\overset{\eqref{eq-epsilon-bound}}{\leq} \frac{m_H^2}{2n}d^2+\frac{nm_H-4m_H}{2}d+1-\epsilon(m_Hd,n)\]
\[+\big(-\frac{m_H^2 d d_m}{n}+\frac{m_H^2d_m^2}{2n}-\frac{m_Hn-4m_H}{2}d_m-1+\epsilon(m_Hd,n)+ g(C_m)+\sum_{j=1}^m k_jm_Hd_j-\frac{m_H^2\sum_{0\leq x,y\leq m-1, x\neq y}d_x'd_y'}{n} \big).\]
Thus, we only need to prove $R'_m\leq 0$, where
\[R'_m:=-\frac{m_H^2 d d_m}{n}+\frac{m_H^2d_m^2}{2n}-\frac{m_Hn-4m_H}{2}d_m-1+\epsilon(m_Hd,n)+ g(C_m)+\sum_{j=1}^m k_jm_Hd_j-\frac{m_H^2\sum_{0\leq x,y\leq m-1, x\neq y}d_x'd_y'}{n}.\]
Note that
\[-\frac{m_H^2 d d_m}{n}+\frac{m_H^2d_m^2}{2n}\overset{\eqref{eq-dj'}}{=}\frac{-m_H^2d_m\sum^{m-1}_{i=0}d_i'}{n}-\frac{m_H^2d_m^2}{2n}\]
and
\[\frac{-m_H^2d_m\sum^{m-1}_{i=0}d_i'}{n}+\sum_{j=1}^m k_jm_Hd_j-\frac{m_H^2\sum_{0\leq x,y\leq m-1, x\neq y}d_x'd_y'}{n}\]
\[\overset{\eqref{eq-dj'}}{=}-\frac{m_H^2\sum^{m-1}_{i=0} d_i'd_{i+1}}{n}+\sum_{j=1}^m k_jm_Hd_j=\sum^{m}_{j=1} m_Hd_j(k_j-\frac{m_Hd'_{j-1}}{n}),\]
then we obtain
\[R'_m=\sum^{m}_{j=1} m_Hd_j(k_j-\frac{m_Hd'_{j-1}}{n})-\frac{m_H^2d_m^2}{2n}-\frac{m_Hn-4m_H}{2}d_m-1+\epsilon(m_Hd,n)+ g(C_m).\]
And from (v) and \eqref{eq-jiaocha-thm}, we get $k_j-\frac{m_Hd'_{j-1}}{n}\leq 0$, then combing with $\epsilon(m_Hd,n)\leq \frac{n^2-n}{8}$ in \eqref{eq-epsilon-bound}, we see

\begin{equation}\label{eq-222}
R'_m\leq m_H(k_1-\frac{m_H d_0'}{n})-\frac{m_H^2d_m^2}{2n}-\frac{m_Hn-4m_H}{2}d_m-1+\frac{n^2-n}{8}+ g(C_m).
\end{equation}

If $d_m<N_4$, then by $g(C_m)\leq g(C_m')$, $\deg_h C_m'=m_Hd_m$ and Lemma \ref{lem-cok-dim0}, the the inequality \eqref{eq-222} gives
\[R'_m\leq m_H(k_1-\frac{m_H d_0'}{n})-\frac{m_H^2d_m^2}{2n}-\frac{m_Hn-4m_H}{2}d_m-1+\frac{n^2-n}{8}+ \frac{(m_Hd_m-1)(m_Hd_m-2)}{2}\]
\[\leq m_H(k_1-\frac{m_H d_0'}{n})+\frac{n^2-n}{8}+\frac{m^2_H N^2_4}{2},\]
hence (v) and \eqref{eq-sum-thm} gives $R'_m\leq 0$.

If $d_m\geq N_4$ and $C_m'$ is of Type I or II, by Claim 1 or applying Lemma \ref{lem-bmt-at-ab} to $I_{C'_m}$ at the point $(a,b)=(0,-\frac{4}{3}n)$, we obtain
\[g(C_m)\leq g(C_m')\leq \frac{m_H^2}{2n}d_m^2+\frac{nm_H-4m_H}{2}d_m+1-\epsilon(m_Hd_m, n).\]
Thus,
\[R'_m\leq m_H(k_1-\frac{m_H d_0'}{n})-\frac{m_H^2d_m^2}{2n}-\frac{m_Hn-4m_H}{2}d_m-1+\frac{n^2-n}{8}+ g(C_m)\]
\[\leq  m_H(k_1-\frac{m_H d_0'}{n})+\frac{n^2-n}{8}-\epsilon(m_Hd_m, n)\]
\[\overset{\eqref{eq-epsilon-bound}}{\leq}  m_H(k_1-\frac{m_H d_0'}{n})+\frac{n^2-n}{8}.\]
Then in this case, $R'_m\leq 0$ follows from (v) and \eqref{eq-sum-thm} as well.

To conclude, in both cases, we have $R'_m\leq 0$ and we can deduce that
\[g(C)\leq \frac{m_H^2}{2n}d^2+\frac{n m_H-4m_H}{2}d+1-\epsilon(m_H d, n).\]
\end{proof}

When $X$ has at most isolated singularities, we obtain:

\begin{theorem}\label{thm-general-bound-isolated}
Let $(X, H)$ be a polarised projective 3-fold with at worst isolated singularities and $m_H$ be an integer such that $m_H H$ is very ample. Then there exists an integer $N_H$ defined in \eqref{eq-sum-thm} such that for any $1$-dimensional closed subscheme $C\subset X$ of $H$-degree $d\geq N_H$, we have
\[g(C)\leq \frac{1}{2n_H}(m_H d)^2+\frac{n_H-4}{2}(m_H d)+1-\epsilon(m_H d, n_H).\]
\end{theorem}

\begin{proof}
Since $X_{\mathrm{sing}}$ is finite, we always have $\gedim'(C)\leq 3$, and the result follows from Theorem \ref{thm-general-bound}.
\end{proof}

\subsection{Projective 3-folds of Picard number one}\label{subsec-pic-1}

We end this section by discussing corollaries of Theorem \ref{thm-general-bound} and Theorem \ref{thm-general-bound-isolated} when $X$ is of Picard number one.


\begin{corollary}\label{cor-pic-rk-1}
Let $(X, H)$ be a polarised factorial projective 3-fold with at worst isolated singularities. Assume that $\NS(X)=\ZZ H$ and $n:=H^3$. Let $s$ be the least integer such that $|sH|\neq \varnothing$ and $m_H$ be an integer such that $m_{H} H$ is very ample.

Then there exists an integer $N_H$ defined in \eqref{eq-sum-thm} such that for any $1$-dimensional closed subscheme $C\subset X$ of $H$-degree $d\geq N_{H}$, we have
\[g(C)\leq \frac{1}{2sn}d^2+\frac{snm^3_{H}-4m_{H}}{2}d+1-\epsilon(m_{H} d, snm^2_{H}).\]

\end{corollary}


\begin{proof}
By Remark \ref{rmk-nH}, we have $n_{H}=snm^2_{H}$, then the result follows from Theorem \ref{thm-general-bound-isolated}.
\end{proof}

\begin{remark}\label{rmk-determine-s}
In practice, the integer $s\geq 1$ can be easily determined. For example, $s=1$ for any complete intersection 3-folds in projective spaces. And by \cite[Theorem 1.5]{horing:effective-vanish}, we also have $s=1$ when $\mathrm{char}(\kk)=0$ and $X$ has at worst canonical singularities and $-K_X$ nef. In particular, all smooth Fano or Calabi--Yau 3-folds over $\CC$ satisfying $s=1$.
\end{remark}

In the situation of Corollary \ref{cor-pic-rk-1}, we can show that this upper bound is asymptotically optimal. Let $\mathsf{D}_{X}$ be the defining domain of $\gmax^{X}(d)$. In other words,
\[\mathsf{D}_{X}:=\{d\in \ZZ_{\geq 1}~|~\text{there is a }1\text{-dimensional closed subscheme }C\subset X \text{ of }\deg(C)=d\}.\]
It is clear that the number in $\mathsf{D}_{X}$ can be arbitrarily large.

\begin{theorem}\label{thm-limit}
Let $(X, H)$ be a polarised factorial projective 3-fold with at worst isolated singularities. Assume that $\NS(X)=\ZZ H$ and $n:=H^3$.  Let $s$ be the least integer such that $|sH|\neq \varnothing$. Then we have
\[\underset{d\in \mathsf{D}_{X}}{\lim_{d\to +\infty}} \frac{\gmax^{X}(d)}{d^2}=\frac{1}{2sn}.\]
\end{theorem}

\begin{proof}
According to Corollary \ref{cor-pic-rk-1}, we have
\[\gmax^{X}(d)\leq \frac{1}{2sn}d^2+\text{linear term}+\text{constant}\]
when $d\gg 0$ and $d\in \mathsf{D}_{X}$. Then to prove the statement, it is enough to find out a function $f(d)\colon \mathsf{D}_{X}\to \mathbb{R}$ such that
\[\gmax^{X}(d)\geq f(d)\]
when $d\gg 0$ and  
\[\underset{d\in \mathsf{D}_{X}}{\lim_{d\to +\infty}} \frac{f(d)}{d^2}=\frac{1}{2sn}.\]
To this end, let $m$ be the least integer such that $mH$ is very ample and $D\in |sH|$. As $\NS(X)=\ZZ H$, the surface $D$ is integral. We define a set
\[\mathsf{S}:=\{l~|~d \equiv l\mod{msn},~0\leq l\leq msn-1,~d\in \mathsf{D}_{X}\}.\]
Let $t=\mathrm{card}(\mathsf{S})$ and  $C_1,\cdots, C_t$ be $1$-dimensional closed subschemes of $X$ such that $\deg(C_i)=d_i$ and
\[\{l_i~|~d_i \equiv l_i\mod{msn},~0\leq l_i\leq msn-1,~1\leq i\leq t\}=\mathsf{S}.\]
Then for any integer $k\geq 1$ and $1\leq i\leq t$, we define a $1$-dimensional closed subscheme $C^k_i$ as $C^k_i:=C_i\cup (D\cap S_k)$, where $S_k\in |kmH|$ is a generic divisor that intersects $D$ and $C_i$ properly. From the construction, we see
\[\deg(C^k_i)=d_i+kmsn\]
and 
\[g(C^k_i)=g(C_i)+\frac{1}{2sn}(kmsn)^2+\frac{s+\frac{K_X.H^2}{n}}{2}(kmsn)+1+\mathrm{length}(C_i\cap (D\cap S_k)).\]
Moreover, by definition, we see
\[\{d_i+kmsn~|~k\geq 1,~1\leq i\leq t\}=\mathsf{D}_{X}\cap \ZZ_{\geq N}\]
for a sufficiently large integer number $N$. Then if we define
\[f(d):=g(C_i)+\frac{1}{2sn}(d-d_i)^2+\frac{s+\frac{K_X.H^2}{n}}{2}(d-d_i)+1+\mathrm{length}(C_i\cap (D\cap S_k))\]
for $d\in \mathsf{D}_{X}\cap \ZZ_{\geq N}$ and $d\equiv d_i \mod{msn}$, and 
\[f(d):=g^{X}_{\max}(d)\]
for $d\in \mathsf{D}_{X}\cap \ZZ_{<N}$, we get a function $f(d)\colon \mathsf{D}_{X}\to \mathbb{R}$ such that
\[f(d)\leq g^{X}_{\max}(d)\]
when $d\gg 0$. Moreover, as 
\[0\leq \mathrm{length}(C_i\cap (D\cap S_k))\leq kmd_i=\frac{d-d_i}{sn}d_i,\]
we obtain
\[g(C_i)+\frac{1}{2sn}(d-d_i)^2+\frac{s+\frac{K_X.H^2}{n}}{2}(d-d_i)+1\leq f(d)\leq g(C_i)+\frac{1}{2sn}(d-d_i)^2+\frac{s+\frac{K_X.H^2}{n}}{2}(d-d_i)+1+\frac{d-d_i}{sn}d_i\]
for $d\gg 0$, which implies 
\[\underset{d\in \mathsf{D}_{X}}{\lim_{d\to +\infty}} \frac{f(d)}{d^2}=\frac{1}{2sn}\]
as desired.
\end{proof}

\begin{remark}
From the proof above, we also obtain $\frac{K_X.H^2}{n}\leq snm_H^3-4m_H-s$.
\end{remark}

If we further assume that $m_H=1$, i.e.~$H$ is very ample, we obtain the following simpler bound.

\begin{corollary}\label{cor-very-ample}
Let $(X, H)$ be a polarised factorial projective 3-fold with at worst isolated singularities. Assume that $\NS(X)=\ZZ H$ with $H$ very ample and $n:=H^3$. 

Then there exists an integer $N_H$ defined in \eqref{eq-sum-thm} such that for any $1$-dimensional closed subscheme $C\subset X$ of $H$-degree $d\geq N_H$, we have
\[g(C)\leq \frac{1}{2n}d^2+\frac{n-4}{2}d+1-\epsilon(d, n).\]
\end{corollary}

\begin{proof}
The statement directly follows from Corollary \ref{cor-pic-rk-1} since $m_H=1$ and $|H|\neq \varnothing$.
\end{proof}

\section{Low-degree curves}\label{sec-low-deg}

In this section, we study low-degree curves on some complete intersection Calabi--Yau 3-folds in projective spaces and a pair of Calabi--Yau 3-folds in the Pfaffian--Grassmannian correspondence. We still assume that the base field $\kk$ is any algebraically closed field.

For a sequence of integers $2\leq k_1\leq k_2\leq \dots \leq k_r$, we denote by $X_{k_1,k_2,\dots, k_r}\subset \PP^{r+3}$ a complete intersection 3-fold of multi-degree $(k_1,k_2,\dots, k_r)$. The following conjecture is a refined version of Conjecture \ref{conj-cast}, which is motivated by the physical predictions in \cite{HKQ09,soheyla:cast}.

\begin{conjecture}[{Optimal Castelnuovo Bound Conjecture}]\label{conj-optimal-bound}
Let $X:=X_{k_1,k_2,\dots, k_r}\subset \PP^{r+3}$ be a smooth complete intersection Calabi--Yau 3-fold over $\CC$ and $n:=k_1k_2\dots k_r$. Then
\[g_{\max}^{X}(d)\leq \frac{1}{2n}d^2+\frac{1}{2}d+1-\epsilon_X(d),\]
where $\epsilon_X(d)=\epsilon_X(f)=\epsilon_X(n-f)$ for $f\equiv d \mod n$ with $1\leq f\leq n$ and $\epsilon_X(n)=0$, satisfying

\begin{itemize}
    \item $X_5$: $\epsilon_X(1)=\frac{8}{5}$ and $\epsilon_X(2)=\frac{12}{5}$,

    \item $X_{2,4}$: $\epsilon_X(1)=\frac{25}{16}$, $\epsilon_X(2)=\frac{9}{4}$, $\epsilon_X(3)=\frac{33}{16}$, and $\epsilon_X(4)=1$,

    \item $X_{3,3}$: $\epsilon_X(1)=\frac{14}{9}$, $\epsilon_X(2)=\frac{20}{9}$, $\epsilon_X(3)=2$, and $\epsilon_X(4)=\frac{26}{9}$,

    \item $X_{2,2,3}$: $\epsilon_X(1)=\frac{37}{24}$, $\epsilon_X(2)=\frac{13}{6}$, $\epsilon_X(3)=\frac{15}{8}$, $\epsilon_X(4)=\frac{8}{3}$, $\epsilon_X(5)=\frac{61}{24}$, and $\epsilon_X(6)=\frac{3}{2}$,

    \item $X_{2,2,2,2}$: $\epsilon_X(1)=\frac{49}{32}$, $\epsilon_X(2)=\frac{17}{8}$, $\epsilon_X(3)=\frac{89}{32}$, $\epsilon_X(4)=\frac{5}{2}$, $\epsilon_X(5)=\frac{105}{32}$, $\epsilon_X(6)=\frac{25}{8}$, $\epsilon_X(7)=\frac{97}{32}$, and $\epsilon_X(8)=2$.
\end{itemize}
\end{conjecture}

We will use our algorithm described in Section \ref{intro-subsec-method} to prove the above conjecture for low-degree curves even when $X$ is singular.

\begin{theorem}\label{thm-low-deg}
Let $X_{k_1,\cdots, k_r}$ be any complete intersection 3-fold in $\PP^{r+3}$ over $\kk$ with at worst isolated singularities, where $2\leq k_1\leq \cdots\leq k_r$. If $X$ contains no planes and quadric surfaces, then Conjecture \ref{conj-optimal-bound} holds for $d\leq D_1$, where

\begin{enumerate}
    \item $X_5$: $D_1=15$,

    \item $X_{2,4}$: $D_1=8$,

    \item $X_{3,3}$: $D_1=9$,

    \item $X_{2,2,3}$: $D_1=6$, and

    \item $X_{2,2,2,2}$: $D_1=6$.
\end{enumerate}

\end{theorem}

In Section \ref{subsec-grass}, we deal with Pfaffian--Grassmannian Calabi--Yau 3-folds. Although there are no general optimal expectations as in Conjecture \ref{conj-optimal-bound}, some tables of GV-invariants are computed by physicists in \cite{Haghighat:2008ut,Hosono:2007vf}. We also verify the corresponding bound for low-degree curves.

By analyzing the geometry of $X$ more carefully, all the results in this section can be improven and extended to other 3-folds $X$. This will be done in a sequel \cite{fey-liu:sharp}.

\subsection{Preliminaries}

Before proving Theorem \ref{thm-low-deg}, we present some lemmas concerning the behavior of curves under projections. These results will be used in Step 3 when we apply the iterative algorithm (cf.~Section \ref{intro-subsec-method}).

We first state a general bound of the genus for space curves.

\begin{lemma}\label{lem-general-bound}
Let $C\subset \PP^n$ be a $1$-dimensional closed subscheme with $\gedim'(C)\leq 3$ and $n>1$. Then
\[g(C)\leq \frac{(\deg(C)-1)(\deg(C)-2)}{2}.\]
and the equality holds if and only if $C\subset \PP^2\subset \PP^n$ and $C$ is a curve.
\end{lemma}

\begin{proof}
The first statement follows from Lemma \ref{lem-cok-dim0}. If $C\subset \PP^2$ is a curve, then it is clear that the equality holds. Thus in the following, we assume the equality holds and aim to prove $C\subset \PP^2\subset \PP^n$ and $C$ is a curve. By \cite[Lemma 4.3]{liu-ruan:cast-bound}, we only need to show $C\subset \PP^2\subset \PP^n$.

When $n=2$, the statement is trivial, so we can assume that $n\geq 3$. If $n=3$, the result follows from \cite[Theorem 3.3]{Hartshorne1994TheGO} or \cite[Theorem 1.8]{liu-ruan:cast-bound}.

If $n>3$, let $\pi\colon \PP^n\dashrightarrow \PP^3$ be a good projection of $C$ and $C':=\pi(C)$. Then $g(C)\leq g(C')$. As
\[g(C')\leq \frac{(\deg(C)-1)(\deg(C)-2)}{2}\]
by the case $n=3$, we see $g(C)=g(C')$ and $C'\subset \PP^2$. Hence from Lemma \ref{lem-cok-dim0}, we know that $\pi$ induces an isomorphism $C\cong C'$ and $C\subset \PP^2\subset \PP^n$ by Remark \ref{rmk-proj-iso}.
\end{proof}

Nest, we relate the existence of some walls after projections to the geometry of curves.

\begin{lemma}\label{lem-wall-deg2-integral}
Let $C\subset \PP^3$ be a $1$-dimensional closed subscheme of degree $d$ such that $d\geq 4$. Assume the uppermost wall of $I_C$ is given by $W(\oh_{\PP^3}(-D), I_C)$. If $D\in |\oh_{\PP^3}(2)|$, then $D$ is integral and normal.
\end{lemma}

\begin{proof}
As $d\geq 4$, we have $b_d\leq -2$ and \begin{equation}\label{eq-sec-5-1}
    \lceil\frac{d}{2}\rceil\geq \lceil\sqrt{2d}-\frac{1}{2}\rceil.
\end{equation}

If $D$ is reducible, then $D=\PP_1\cup \PP_2$, where $\PP_1$ and $\PP_2$ are two planes in $\PP^3$. As $C\subset \PP_1\cup \PP_2$, we see $\deg(C\cap \PP_1)+\deg(C\cap \PP_2)\geq d$. Hence by relabeling, we can assume that $\deg(C\cap \PP_1)\geq \lceil\frac{d}{2}\rceil$. Then by \eqref{eq-sec-5-1} and Lemma \ref{lem-construct-wall}, $W(I_C, I_{C\cap \PP_1/\PP_1})$ is an actual wall for $I_C$ above $W(\oh_{\PP^3}(-D), I_C)$, contradicts our assumption.

If $D$ irreducible but not reduced, then $D$ is a double plane and $D_{\mathrm{red}}=\PP_1$, where $\PP_1\subset \PP^3$ is a plane. Therefore, we have an exact sequence $0\to I_{Y}(-1)\to I_{C}\to I_{C\cap \PP_1/\PP_1}\to 0$ where $Y\subset C\cap \PP_1$ is the residue scheme to the intersection of $C$ with $D$ as in \cite[pp.3]{hartshorne:double-plane}. Hence, we have $\deg(Y)+\deg(C\cap \PP_1)=d$, which implies $\deg(C\cap \PP_1)\geq \lceil\frac{d}{2}\rceil$ since $Y\subset C\cap \PP_1$. Now using \eqref{eq-sec-5-1} and Lemma \ref{lem-construct-wall} again, we get a contradiction.

Therefore, $D$ is integral in each case and it is automatically normal by \cite[Corollary 3.4.19]{lazar:positivity-I}.
\end{proof}

\begin{lemma}\label{lem-not-in-plane}
Let $C\subset \PP^n$ be a $1$-dimensional closed subscheme with $\gedim(C)\leq 3$. Then either

\begin{enumerate}
    \item $C\subset \PP^2\subset \PP^n$, or

    \item $\pi(C)$ is not contained in a plane for any generic good projection $\pi\colon \PP^n\dashrightarrow \PP^3$ of $C$.
\end{enumerate}
\end{lemma}

\begin{proof}
When $n\leq 3$, the statement is trivial, so we can only consider $n>3$. Suppose that $C\subset \PP^n$ is not contained in any $\PP^2\subset \PP^n$.

Assume for a contradiction that $\pi(C)$ is contained in a plane for any generic good projection $\pi\colon \PP^n\dashrightarrow \PP^3$ of $C$. Let $\pi\colon \PP^n\dashrightarrow \PP^3$ is a good projection of $C$ such that $\pi(C)\subset \PP^2\subset \PP^n$, then we see $C\subset \overline{\pi^{-1}(\PP^2)}=\PP^{n-1}\subset \PP^n$. So for any $p\in \PP^n\setminus \overline{\pi^{-1}(\PP^2)}$, the projection from $p$ induces an isomorphism of $\overline{\pi^{-1}(\PP^2)}$ onto its image, and we can assume that the first projection of $\pi$ is an isomorphism when restricted to $C$. Next, by our assumption, there is a good projection $\pi'\colon \PP^{n-1}\dashrightarrow \PP^3$ of $C\subset \PP^{n-1}$ such that $\pi'(C)\subset \PP^2$. Then the same argument as above shows that $C\subset \PP^{n-2}\subset \PP^{n-1}$ and we can assume the first projection of $\pi'$ is an isomorphism when restricted to $C$. So if we continue this process, we finally get $C\subset \PP^3$ and is contained in a plane, a contradiction.
\end{proof}

\begin{lemma}\label{lem-3-proj}
Let $C\subset \PP^4$ be a curve of degree $d\geq 4$ with $\gedim(C)\leq 3$ and $p_1,p_2,p_3$ be three different good projections of $C$. Assume that the uppermost actual wall of $p_i(C)$ is given by $\oh_{\PP^3}(2)$ for each $i=1,2,3$. Then

\begin{enumerate}
    \item $C$ is contained in a $(2,2)$-complete intersection surface $S\subset \PP^4$ such that $\gedim(S)\leq 3$, and

    \item in the situation of (1), if furthermore $p_3$ is a good projection of both $C$ and $S$ such that the uppermost actual wall for $p_3(C)$ is given by $\oh_{\PP^3}(2)$, then either $C$ is contained in a degree $2$ irreducible surface in $S$, or $C$ is contained in a $(2,2,2)$-complete intersection curve in $\PP^4$.
\end{enumerate}

\end{lemma}

\begin{proof}
By the assumption, we know that there exist $D_1, D_2, D_2\in |\oh_{\PP^3}(2)|$ such that $p_i(C)\subset D_i$ for each $1\leq i\leq 3$. From Lemma \ref{lem-wall-deg2-integral}, we see each $D_i$ is integral and normal, hence $\overline{p_i^{-1}(D_i)}\in |\oh_{\PP^4}(2)|$ is integral and normal as well. In particular, $\dim \mathrm{Sing}(\overline{p_i^{-1}(D_i)})\leq 1$ for each $i$. Then the degree $4$ surface $S:=\overline{p_1^{-1}(D_1)}\cap \overline{p_2^{-1}(D_2)}\subset \PP^4$ satisfies $\gedim(S)\leq 3$ which proves (a).

For part (b), if $D_3$ intersects $p_3(S)$ properly, then $D_3\cap p_3(S)$ is a degree $8$ curve containing $p_3(C)$. Hence $C\subset \overline{p_i^{-1}(D_3)}\cap S$. As $p_3|_S$ is birational onto $p_3(S)$, we know that $\overline{p_i^{-1}(D_3)}\cap S$ is of dimension $1$. Hence $\overline{p_i^{-1}(D_3)}\cap S=\overline{p_i^{-1}(D_1)}\cap \overline{p_i^{-1}(D_2)}\cap \overline{p_i^{-1}(D_3)}$ is a $(2,2,2)$-complete intersection curve. If $D_3$ does not intersect $p_3(S)$ properly, as $D_3$ is integral, we see $D_3\subset S$. This implies $C\subset (p_3|_S)^{-1}(D_3)$, which is a degree $2$ irreducible surface in $S$.
\end{proof}

\begin{lemma}\label{lem-diff-g}
Let $C\subset \PP^n$ be a $1$-dimensional closed subscheme such that $\gedim(C)\leq 3$ and $n>3$. If $\pi\colon \PP^n\dashrightarrow \PP^3$ is a good projection of $C$ such that
\[g(C)\geq g(\pi(C))-t\]
for $0\leq t\leq n-4$, then $C\subset \PP^{3+t}\subset \PP^n$.
\end{lemma}

\begin{proof}
As $g(C)\leq g(\pi(C))$ by Lemma \ref{lem-cok-dim0}, we know that there is a subset $I\subset \{1,2,\dots, n-3\}$ with $\mathrm{card}(I)\geq n-3-t$ such that $\pi_i|_{C_{i-1}}$ is an isomorphism for each $i\in I$, where $C_0:=C$ and $C_{i-1}:=\pi_{i-1}\circ\cdots \circ \pi_1(C)$ for $2\leq i\leq n-3$. Hence $C$ is contained in a linear subspace of $\PP^n$ of codimension at least $n-3-t$, i.e.~$C\subset \PP^{3+t}\subset \PP^n$.
\end{proof}

\subsection{Complete intersection Calabi--Yau 3-folds}

Now we are ready to prove Theorem \ref{thm-low-deg}. Although the argument is similar, we present the proof for each case separately as the geometry is different.

Let $C\subset X$ be a curve of degree $d$ and genus $g$ and $$\pi\colon X\to \PP^3$$ be a good projection of $C$. Let $t$ be the minimal number such that $\pi(C)$ is $(t)$-neutral and $b_d\leq t\leq -1$. Then the uppermost wall in the range $b\in [b_d, -1]$, if it exists, tangents with the line $b=t$.

\subsubsection{$X_5$}

\begin{proof}[{Proof of Theorem \ref{thm-low-deg}(a)}]

When $d\leq 5$, the result follows from Lemma \ref{lem-general-bound}, so we can assume that $d\geq 6$. Since $X$ contains no planes, the intersection of $X$ with any plane is contained in a planar curve of degree $5$. So $C$ can not be planar, and we can take $\pi$ such that $\pi(C)$ is not contained in any plane by Lemma \ref{lem-not-in-plane}, i.e.~$W(\oh_{\PP^3}(-1), I_{\pi(C)})$ is not an actual wall for $\pi(C)$. Therefore, we can assume that $t<-1$.


Note that when $6\leq d\leq 15$ and $t=b_d$, applying Lemma \ref{lem-bmt-at-ab} to $\pi(C)$ directly gives the desired bound in Theorem \ref{thm-low-deg}(a), so we may assume that $b_d<t$.

\begin{itemize}
    \item $6\leq d\leq 8$: If $-2<t$, then the uppermost wall of $\pi(C)$ is given by $W(I_{C_1}(-1), I_{\pi(C)})$ for a curve $C_1\subset \pi(C)$ such that $d_1\leq 3$ when $d=6,7$, and $d_1\leq 4$ when $d=8$ (cf.~Proposition \ref{cor_wall}). Then applying Proposition \ref{prop-decompose-wall} and the result of Theorem \ref{thm-low-deg}(a) for lower-degree subschemes, we can conclude that Theorem \ref{thm-low-deg}(a) holds for $C$.

    If $t=-2$, the uppermost wall of $\pi(C)$ is given by $W(\oh_{\PP^3}(-2), I_{\pi(C)})$. Then the result follows from Lemma \ref{lem-ms-bound}.

    If $t<-2$, then $-3<b_d< t<-2$. By Proposition \ref{cor_wall}, the uppermost wall of $\pi(C)$ is given by $W(I_{C_1}(-2), I_{\pi(C)})$ for a curve $C_1\subset \pi(C)$ such that $d_1\leq 1$. Then the result also follows from applying Proposition \ref{prop-decompose-wall} and the result of Theorem \ref{thm-low-deg}(a) for lower-degree subschemes.

    \item $d=9$: In this case, we need to show $g\leq 12$. When $-3<t\leq -2$, the argument is the same as the case $6\leq d\leq 8$ by using Lemma \ref{lem-ms-bound}, Proposition \ref{prop-decompose-wall}, and the result of Theorem \ref{thm-low-deg}(a) for lower-degree subschemes. 

    When $-2<-t<-1$, the uppermost wall of $\pi(C)$ is given by $W(I_{C_1}(-1), I_{\pi(C)})$ for a curve $C_1\subset \pi(C)$ such that $d_1\leq 5$. If $d_1\leq 4$, the result still follows from Proposition \ref{prop-decompose-wall} and the result of Theorem \ref{thm-low-deg}(a) for lower-degree. However, when $d_1=5$, this only gives $g\leq 13$. 
    
    To fix this, we assume that $g(C)=13$ and try to find a contradiction. Recall that by Proposition \ref{cor_wall}, we have $g(\pi(C))=g(C_1)+g(C_2)+d_1-1$ and $d_1=5$, $d_2=4$. From Lemma \ref{lem-general-bound}, we get $g(C_1)\leq 6$ and $g(C_2)\leq 3$, hence $g(\pi(C))\leq 13$. So we have $g(\pi(C))=13$ and $g(C_1)=6$ by the assumption. But Lemma \ref{lem-general-bound} implies that $C_1\subset P\cong \PP^2\subset \PP^3$. In other words, $\deg(\pi(C)\cap P)\geq 5$, which means $W(I_{C_1}(-1), I_{\pi(C)})$ is not the uppermost wall by Lemma \ref{lem-construct-wall}. This gives a contradiction.
    

    \item $10\leq d\leq 13$ or $d=15$: We can assume that $t\neq -2$. Indeed, if $d=10$, then Lemma \ref{lem-ms-bound} gives $g\leq 16$ as desired. If $11\leq d\leq 13$ and for any good projection $\pi$ of $C$, we have $t=-2$, then by Lemma \ref{lem-3-proj}, $C$ is either contained in a quadric surface $S$ or a $(2,2,2)$-complete intersection curve. Since $S\cap X$ is a curve of degree $10$, we get a contradiction for $11\leq d\leq 13$.
    
    Now we have $b_d<t<-1$ and $t\neq -2$. If $t=-3$, applying Lemma \ref{lem-ms-bound} gives the result. For other values of $t$, we directly use Proposition \ref{prop-decompose-wall} and the result of Theorem \ref{thm-low-deg}(a) for lower-degree as in previous cases to conclude the result.

    \item $d=14$: We need to show $g(C)\leq 26$ in this case. The argument is the same as the above case, except when $-3<t<-2$.
    
    When $-3<t<-2$, we have $d_1\leq 5$. The problem only occurs when $d_1=5$. In this situation, using Proposition \ref{prop-decompose-wall}, we have a $1$-dimensional closed subscheme $C_1'\subset C$ of degree $d_1$ and a divisor $S\in |\oh_X(2)|$ with 
    \[g(C)=g(C_1')+g(C\cap S)+9.\]
    By Lemma \ref{lem-general-bound}, we obtain $g(C_1')\leq 6$. If $g(C_1')\leq 5$, we get $g(C)\leq 26$ from Theorem \ref{thm-low-deg}(a) for lower-degree. If $g(C_1')=6$, then $C_1'\subset P\cong \PP^2\subset \PP^4$ by Lemma \ref{lem-general-bound}, and we have $\deg(C\cap P)\geq d_1=5$. Then by replacing $\pi$ with another generic projection, we can assume that $\pi$ induces an isomorphism $P\cong \pi(P)$. By Lemma \ref{lem-construct-wall}, we back to the case $-2<t<-1$ and the result follows from Proposition \ref{prop-decompose-wall} and the result of Theorem \ref{thm-low-deg}(a) for lower-degree closed subschemes.
 
\end{itemize}

\end{proof}


\subsubsection{$X_{2,4}$}

\begin{proof}[{Proof of Theorem \ref{thm-low-deg}(b)}]

When $d\leq 4$, the result follows from Lemma \ref{lem-general-bound}. So we can assume that $d\geq 5$. Since $X$ contains no planes, the intersection of $X$ with any plane is contained in a planar curve of degree $4$. So $C$ can not be planar, and we can take $\pi$ such that $\pi(C)$ is not contained in any plane by Lemma \ref{lem-not-in-plane}, i.e.~$W(\oh_{\PP^3}(-1), I_{\pi(C)})$ is not an actual wall for $\pi(C)$. Therefore, we can assume that $t<-1$.

\begin{itemize}
    \item $5\leq d\leq 8$: If $t=b_d$, applying Lemma \ref{lem-bmt-at-ab} to $\pi(C)$ directly gives the desired bound in Theorem \ref{thm-low-deg}(b).
    
    If $t=-2$, the uppermost wall of $\pi(C)$ is given by $W(\oh_{\PP^3}(-2), I_{\pi(C)})$. Then the result follows from Lemma \ref{lem-ms-bound}. If $-2<t<-1$, the uppermost wall of $\pi(C)$ is given by $W(I_{C_1}(-1), I_{\pi(C)})$ for a curve $C_1\subset \pi(C)$ such that $d_1\leq 2$ when $d=5$, $d_1\leq 3$ when $d=6,7$, and $d_1\leq 4$ when $d=8$ (cf.~Proposition \ref{cor_wall}). Then applying Proposition \ref{prop-decompose-wall} and the result of Theorem \ref{thm-low-deg}(b) for lower-degree subschemes, we can conclude that Theorem \ref{thm-low-deg}(b) holds for $C$. The argument for $-3<b_d< t<-2$ is also similar.

\end{itemize}

\end{proof}

\subsubsection{$X_{3,3}$}

\begin{proof}[{Proof of Theorem \ref{thm-low-deg}(c)}]

When $d\leq 3$, the result follows from Lemma \ref{lem-general-bound}. So we can assume that $d\geq 4$. Since $X$ contains no planes, the intersection of $X$ with any plane is contained in a planar curve of degree $3$. So $C$ can not be planar, and we can take $\pi$ such that $\pi(C)$ is not contained in any plane by Lemma \ref{lem-not-in-plane}. Therefore, we may further assume that $t<-1$.

Note that when $d\leq 9$ and $t=b_d$, applying Lemma \ref{lem-bmt-at-ab} to $\pi(C)$ directly gives the desired bound in Theorem \ref{thm-low-deg}(a), so we may assume that $b_d<t$.

\begin{itemize}
    \item $4\leq d \leq 6$: If $t=-2$, the uppermost wall of $\pi(C)$ is given by $W(\oh_{\PP^3}(-2), I_{\pi(C)})$. Then the result follows from Lemma \ref{lem-ms-bound}. If $-2<t<-1$, the uppermost wall of $\pi(C)$ is given by $W(I_{C_1}(-1), I_{\pi(C)})$ for a curve $C_1\subset \pi(C)$ such that $d_1\leq 1$ when $d=4$, $d_1\leq 2$ when $d=5$, and $d_1\leq 3$ when $d=6$ (cf.~Proposition \ref{cor_wall}). Then applying Proposition \ref{prop-decompose-wall} and the result of Theorem \ref{thm-low-deg}(c) for lower-degree subschemes, we can conclude that Theorem \ref{thm-low-deg}(c) holds for $C$. The argument for $-3<b_d< t<-2$ is also similar. 

    \item $7\leq d\leq 9$: If after replacing $\pi$ with a generic good projection, we can assume that $t\neq 2$, then the result follows from the same argument as above by using Proposition \ref{prop-decompose-wall} and Theorem \ref{thm-low-deg}(c) for lower-degree.

    Now assume that $t=-2$ for any generic good projection of $C$. If $g(C)= g(\pi(C))$, then $C\subset X\cap \PP^3\subset \PP^5$ by Lemma \ref{lem-diff-g}. As $t=-2$ and $X$ does not contain quadric surfaces, we see that $C\subset S\cap X\subset \PP^3\cap X$ for a quadric surface $S\subset \PP^3$. But this makes a contradiction since $S\cap X$ is a curve of degree $6$. Hence, we have $g(C)\leq g(\pi(C))-1$, which by Lemma \ref{lem-ms-bound} gives $g(C)\leq 5$ for $d=7$.
    
    
    If $8\leq d\leq 9$, we claim that we can assume $g(C)\leq g(\pi(C))-2$. If $g(C)=g(\pi(C))-1$, from Lemma \ref{lem-diff-g} we have $C\subset X\cap \PP^4\subset \PP^5$. By Lemma \ref{lem-3-proj}, we have either $C\subset S\subset \PP^4$ for a degree $2$ surface $S$, or $C\subset C'$ for a $(2,2,2)$-complete intersection curve $C'\subset \PP^4$. The former case is impossible since $S\cap X$ is contained in a curve of degree $6$. And the latter case is possible only when $d=8$, which implies $C=C'$ and $g(C)\leq 5$ since $C$ and $C'$ are curves of the same degree.
    
    Now we have $g(C)\leq g(\pi(C))-2$. Then the result follows from $t=-2$ and applying Lemma \ref{lem-ms-bound}.

\end{itemize}

\end{proof}

\subsubsection{$X_{2,2,3}$}

\begin{proof}[{Proof of Theorem \ref{thm-low-deg}(d)}]

When $d\leq 3$, the result follows from Lemma \ref{lem-general-bound}. Now assume that $4\leq d\leq 6$. Since $X$ contains no planes, the intersection of $X$ with any plane is contained in a planar curve of degree $3$. So $C$ can not be planar, and we can take $\pi$ such that $\pi(C)$ is not contained in any plane by Lemma \ref{lem-not-in-plane}, i.e.~$W(\oh_{\PP^3}(-1), I_{\pi(C)})$ is not an actual wall for $\pi(C)$. Therefore, we can assume that $t<-1$. Then the rest of the argument is the same as the proof of Theorem \ref{thm-low-deg}(b) above.

\end{proof}

\subsubsection{$X_{2,2,2,2}$}

\begin{proof}[{Proof of Theorem \ref{thm-low-deg}(e)}]
When $d\leq 2$, the result follows from Lemma \ref{lem-general-bound}. When $d=3$, we have $g(C)\leq 1$. If $g(C)=1$, then $C\subset P\cong \PP^2\subset \PP^7$ by Lemma \ref{lem-general-bound}. However, as $X$ does not contain a plane, $P\cap X$ is a degree $2$ curve, which makes a contradiction. Hence $g(C)\leq 0$.

Now we assume that $3\leq d\leq 6$. As in the above cases, we can take $\pi$ such that $\pi(C)$ is not contained in any plane by Lemma \ref{lem-not-in-plane}. Hence, we can further assume that $t<-1$.

\begin{itemize}
    \item $d=4$: If $-2<t<-1$, the result follows from Proposition \ref{prop-decompose-wall} and the result of Theorem \ref{thm-low-deg}(e) for lower-degree cases. If $t=-2$, then applying Lemma \ref{lem-ms-bound} gives the desired bound.

    \item $5\leq d\leq 6$: First, we assume that $d=5$. In this case, we need to prove $g(C)\leq 1$. If $g(C)=g(\pi(C))$, by Lemma \ref{lem-diff-g} we have $C\subset \PP^3\cap X\subset \PP^7$. As $\PP^3\cap X$ is contained is a $(2,2)$-complete intersection curve in $\PP^3$, we get a contradiction. Thus, we see $g(C)\leq g(\pi(C))-1$.

    If $t=-2$, then by Lemma \ref{lem-ms-bound} we get $g(C)\leq g(\pi(C))-1\leq 1$. If $-2<t<-1$, then the result follows from Proposition \ref{prop-decompose-wall} and Theorem \ref{thm-low-deg}(e) for lower-degree. If $b_d\leq t<-2$, then applying Lemma \ref{lem-bmt-at-ab}, we get $g(C)\leq g(\pi(C))-1\leq 1$ as well.

    The argument for $d=6$ is similar.
\end{itemize}

\end{proof}

\begin{remark}\label{rmk-weak-bound}
As in Theorem \ref{thm-low-deg}, we can continue the computations for curves of degree $>D_1$. It is not hard to see that similar arguments as in the above cases imply Conjecture \ref{conj-cast} when $d\leq D_2$, where $D_2=D_1+3$. We will continue these computations in a sequel \cite{fey-liu:sharp}.
\end{remark}

\subsection{Pfaffian--Grassmannian Calabi--Yau 3-folds}\label{subsec-grass}

In this subsection, we consider the pair of Calabi--Yau 3-folds in the Pfaffian--Grassmannian correspondence. 

More precisely, we take $X$ to be a smooth codimension $7$ linear section of $\Gr(2,7)$, which is a Calabi--Yau 3-fold. On the other hand, there is another associated Calabi--Yau 3-fold $Y$ given by a smooth complete intersection of the Pfaffian variety $\mathrm{Pf}(7)\subset \PP^{20}$ of $\Gr(2,7)$ with $\PP^6$. It is known that $X$ and $Y$ are derived equivalent but not birational (cf.~\cite{calda:pfaffian}).

Although there is no complete expectation of an optimal bound of $\gmax$ in these two cases as Conjecture \ref{conj-optimal-bound}, the low-genus GV-invariants of $X$ and $Y$ are computed in \cite{Hosono:2007vf,Haghighat:2008ut} via physical approaches and we can verify the corresponding bound of $\gmax$ in low-degree.

\begin{proposition}
Let $X\subset \Gr(2,7)$ be a smooth codimension $7$ linear section. Then $\gmax^X(d)\leq 0$ when $1\leq d\leq 3$ and $\gmax^X(d)\leq 1$ when $4\leq d\leq 5$.
\end{proposition}

\begin{proof}
We use the natural embedding $X\hookrightarrow \Gr(2,7)\hookrightarrow \PP^{20}$. When $1\leq d\leq 2$, the result follows from Lemma \ref{lem-general-bound}. When $d=3$, we have $\gmax^X(3)\leq 1$, and the equality holds only for planar cubic curves by Lemma \ref{lem-general-bound}. As $X$ does not contain any plane and $\Gr(2,7)$ is an intersection of quadric hypersurfaces in $\PP^{20}$, any planar curve in $X$ is contained in a curve of degree $2$. Hence, we get $\gmax^X(3)\leq 0$.

Now let $C\subset X$ be a curve of degree $d$. We fix a good projection $\pi\colon X\to \PP^3$ of $C$. Let $t$ be the minimal number such that $\pi(C)$ is $(t)$-neutral and $b_d\leq t\leq -1$. The argument above shows that $C$ is not contained in any plane when $d\geq 3$, hence $t\neq -1$ by Lemma \ref{lem-not-in-plane}.

\begin{itemize}
    \item If $d=4$, we have $b_d=-2\leq t<-1$. When $t=-2$, Lemma \ref{lem-bmt-at-ab} gives $g(\pi(C))\leq 1$. When $-2<t<-1$, the uppermost wall of $\pi(C)$ is given by $W(I_{C_1}(-1), I_{\pi(C)})$ for a curve $C_1\subset \pi(C)$ of degree $d_1\leq 1$. Then Proposition \ref{prop-decompose-wall} and the result for low-degree curves give $g(C)\leq 1$. 

    \item Now we assume that $d=5$. We claim that $g(C)< g(\pi(C))$. Indeed, if $g(C)= g(\pi(C))$, Lemma \ref{lem-diff-g} implies that $C\subset X\cap \PP^3$. As $\Gr(2,7)$ is an intersection of quadrics and $\dim  X\cap \PP^3\leq 1$, we know that $ X\cap \PP^3$ is contained in a $(2,2)$-complete intersection curve in $\PP^3$, which is of degree $4$ and contradicts $C\subset X\cap \PP^3$.
    
    If $t\leq -2$, then by applying Lemma \ref{lem-bmt-at-ab} to $(a,b)=(0,-2)$, we have $g(\pi(C))\leq 2$, which gives $g(C)\leq 1$. If $-2<t<-1$, the uppermost wall of $\pi(C)$ is given by $W(I_{C_1}(-1), I_{\pi(C)})$ for a curve $C_1\subset \pi(C)$ of degree $d_1\leq 2$. Then Proposition \ref{prop-decompose-wall} and the result for low-degree curves imply $g(C)\leq 1$.
\end{itemize}
\end{proof}

\begin{proposition}
    Let $Y\subset \PP^6$ be a smooth codimension $14$ linear section of $\mathrm{Pf}(7)\subset \PP^{20}$. Then $\gmax^Y(d)\leq 0$ when $1\leq d\leq 2$, $\gmax^Y(d)\leq 1$ when $3\leq d\leq 4$, and $\gmax^Y(5)\leq 2$.
\end{proposition}

\begin{proof}
When $1\leq d\leq 3$, the result follows from Lemma \ref{lem-general-bound}. Now let $C\subset Y$ be a curve of degree $d\geq 4$. We fix a good projection $\pi\colon Y\to \PP^3$ of $C$. Let $t$ be the minimal number such that $\pi(C)$ is $(t)$-neutral and $b_d\leq t\leq -1$. Note that $Y$ is an intersection of cubic hypersurfaces in $\PP^6$ (cf.~\cite[Section 2.1]{Hosono:2007vf}), hence $C$ is not contained in any plane and $t\neq -1$ by Lemma \ref{lem-not-in-plane}.

\begin{itemize}
    \item If $d=4$, we have $b_d=-2\leq t<-1$. When $t=-2$, Lemma \ref{lem-bmt-at-ab} gives $g(\pi(C))\leq 1$. When $-2<t<-1$, the uppermost wall of $\pi(C)$ is given by $W(I_{C_1}(-1), I_{\pi(C)})$ for a curve $C_1\subset \pi(C)$ of degree $d_1\leq 1$. Then Proposition \ref{prop-decompose-wall} and the result for low-degree curves give $g(C)\leq 1$.

    \item Now we assume that $d=5$. If $t\leq -2$, then by applying Lemma \ref{lem-bmt-at-ab} to $(a,b)=(0,-2)$, we have $g(C)\leq g(\pi(C))\leq 2$. If $-2<t<-1$, the uppermost wall of $\pi(C)$ is given by $W(I_{C_1}(-1), I_{\pi(C)})$ for a curve $C_1\subset \pi(C)$ of degree $d_1\leq 2$. Then Proposition \ref{prop-decompose-wall} and the result for low-degree curves imply $g(C)\leq 2$.
\end{itemize}
\end{proof}

\section{Applications} \label{sec_app}

In this section, we discuss some applications of our results. We continue to assume that $\kk$ is an algebraically closed field of any characteristic, except $\kk=\CC$ in Section \ref{subsec-vanish}.

\subsection{Higher-dimensional varieties}\label{subsec-high-dim}

By taking general hyperplane sections, it is not hard to generalize the results in Section \ref{subsec-pic-1} to higher-dimensional varieties. Recall that for a smooth projective variety $X$ and a $1$-dimensional closed subscheme $C\subset X$, we have
\[g(C)=1+\frac{K_X.C}{2}+\ch_{\dim X}(I_C)=1+\frac{K_X.C}{2}-\ch_{\dim X}(\oh_C).\]
On the other hand, if $H$ is a very ample divisor and $Z\subset X$ is a codimension two closed subscheme, then
\[\ch_{\dim X-1}(\oh_{Z'})=Z.H^{\dim X-3},\quad \ch_{\dim X}(\oh_{Z'})=\ch_3(\oh_Z).H^{\dim X-3}-\frac{\dim X-3}{2}Z.H^{\dim X-2},\]
where $Z'$ is the intersection of $Z$ with $X'$ such that $X'\subset X$ is an intersection of $\dim X-3$ general divisors in $|H|$. When $\NS(X)=\ZZ H$, by Grothendieck--Lefschetz theorem \cite[Expos\'e XII, Corollaire 3.6]{sga2}, we have $\NS(X')=\ZZ H|_{X'}$. Applying Corollary \ref{cor-very-ample} to $Z'\subset X'$, we obtain:

\begin{corollary}
Let $(X, H)$ be a polarised smooth projective variety. Assume that $\NS(X)=\ZZ H$ with $H$ very ample and $n:=H^{\dim X}$. 

Then there exist an integer $N_{H}$ defined in \eqref{eq-sum-thm} such that for any codimension $2$ closed subscheme $Z\subset X$ with $d:=Z.H^{\dim X-2}\geq N_{H}$, we have
\[\ch_3(I_Z).H^{\dim X-3}+\frac{K_X.Z.H^{\dim X-3}}{2}\leq \frac{1}{2n}d^2+\frac{n-\dim X-1}{2}d-\epsilon(d, n).\]
\end{corollary}

In particular, when $X$ is a Calabi--Yau 4-fold, we have:

\begin{corollary}\label{cor-bound-cy4}
Let $(X, H)$ be a polarised projective Calabi--Yau 4-fold. Assume that $\NS(X)=\ZZ H$ with $H$ very ample and $n:=H^{4}$. 

Then there exist an integer $N_{H}$ defined in \eqref{eq-sum-thm} such that for any $2$-dimensional closed subscheme $Z\subset X$ with $d:=Z.H^{2}\geq N_{H}$, we have
\[\ch_3(I_Z).H\leq \frac{1}{2n}d^2+\frac{n-5}{2}d-\epsilon(d, n).\]
\end{corollary}

When $X\subset \PP^5$ is a sextic 4-fold, we expect that the bound in Corollary \ref{cor-bound-cy4} is optimal.

Other results in Section \ref{subsec-pic-1} can be easily generalized to higher dimensions as well. We do not state them here since we will not use them in the rest of the paper.

\subsection{Vanishing of enumerative invariants}\label{subsec-vanish}

Next, we discuss some applications in enumerative geometry. In this subsection, we set $\kk=\CC$.

\subsubsection{Curve-counting invariants of Calabi--Yau 3-folds}

Let $X$ be a projective Calabi--Yau 3-fold. Recall that a two-term complex of coherent sheaves on $X$
\[[\oh_X\to F]\]
is called a \emph{stable pair} if $F$ is pure one-dimensional and the cokernel of the above map is zero-dimensional.

Let $C$ be the scheme-theoretic support of $F$ and $\mathrm{H}_2(X, \mathbb{Z})_{c}$ be the subset of $\mathrm{H}_2(X, \mathbb{Z})$ consists of curve classes. For any integer $s$ and curve class $\beta\in \mathrm{H}_2(X, \mathbb{Z})_c$, we denote by $P_s(X, \beta)$ the moduli space of stable pairs $[\oh_X\to F]$
on $X$ with
\[\chi(F)=s,\quad[C]=\beta.\]
It is shown in \cite{PT07} that the moduli space $P_s(X, \beta)$ is a projective scheme with a virtual cycle $[P_s(X, \beta)]^{vir}$ of virtual dimension zero. We define a \emph{Pandharipande--Thomas (PT) invariant} by
\[\PT_{s,\beta}:=\int_{[P_s(X, \beta)]^{vir}} 1.\]
The generating series of PT-invariants is
\[\PT(q, t):=1+\sum_{\beta\neq 0}\sum_{s\in \ZZ} \PT_{s, \beta} q^s t^{\beta}.\]

Recall that for any non-zero effective 1-cycle $\beta\in \mathrm{H}_2(X,\ZZ)$ and $g\geq 0$, the \emph{Gopakumar--Vafa (GV) invariant} $\GV_{g,\beta}$ is defined as the coefficient in an expansion of the series
\begin{equation} \label{GV_from_GW}
    \GW(\lambda, t)=\sum_{g\geq 0, \beta\neq 0} \GW_{g,\beta}\lambda^{2g-2}t^{\beta}=\sum_{g\geq 0,r\geq 1, \beta\neq 0}  \frac{\GV_{g,\beta}}{r}\cdot (2\sin(\frac{r\lambda}{2}))^{2g-2}\cdot t^{r\beta},
\end{equation}
where $\GW_{g,\beta}$ is the Gromov--Witten invariant of $X$ of genus $g$ and class $\beta$.

As in \cite[Section 3]{PT07}, we can write the connected PT series $\sF_{P}(q, t)$ uniquely as
\begin{equation} \label{GV_PT_1030}
    \sF_{P}(q, t):=\log \PT(q, t)=\sum_{g>-\infty} \sum_{\beta\neq 0} \sum_{r\geq 1} \GV'_{g,\beta} \frac{(-1)^{g-1}}{r} ((-q)^{\frac{r}{2}}-(-q)^{-\frac{r}{2}})^{2g-2} t^{r\beta},
\end{equation}
such that $\GV'_{g,\beta}=0$ for a fixed $\beta$ and $g\gg 0$.

By \cite{pardon:mnop}, the GW/PT correspondence \cite{mnop:I,mnop:II} holds for $X$. After changing the variable $q=-\exp(i\lambda)$, the parts of series \eqref{GV_from_GW} and \eqref{GV_PT_1030} containing $t^{\beta}$ are equal as rational functions in $q$ for any $\beta\neq 0$. Now by \cite{DIW21}, we know that $\GV_{g,\beta}=0$ for any fixed $\beta$ and $g\gg 0$. Thus from the uniqueness of the expansions \cite[Lemma 3.9, 3.11]{PT07}, we have
\[\GV'_{g,\beta}=\GV_{g,\beta},\]
or in other words,
\begin{equation} \label{GV_PT_5}
    \sF_{P}(q, t)=\log \PT(q, t)=\sum_{g\geq 0} \sum_{\beta\neq 0} \sum_{r\geq 1} \GV_{g,\beta} \frac{(-1)^{g-1}}{r} ((-q)^{\frac{r}{2}}-(-q)^{-\frac{r}{2}})^{2g-2} t^{r\beta}.
\end{equation}

On the other hand, in \cite{thomasDT}, Thomas constructed a zero-dimensional virtual cycle $[M_{s,\beta}(X)]^{vir}$ on the moduli space $M_{s,\beta}(X)$ of closed subschemes $Z\subset X$ with $[Z]=\beta$ and $\chi(\oh_Z)=s$. Then we define a \emph{Donaldson--Thomas (DT) invariant} by
\[\DT_{s, \beta}:=\int_{[M_{s,\beta}(X)]^{vir}} 1.\]
When the torsion-free part of $\mathrm{H}_2(X, \ZZ)$ is isomorphic to $\ZZ$ (e.g.~$X$ is of Picard number one), we simplify the notations as  $\PT_{s,d}:=\PT_{s,\beta}$ and $\DT_{s,d}:=\DT_{s,\beta}$, where $d$ is the image of $\beta\in \mathrm{H}_2(X, \ZZ)$ in $\ZZ$.

As a consequence of Corollary \ref{cor-pic-rk-1}, we prove:

\begin{theorem}\label{thm-vanishing}
Let $X$ be a projective Calabi--Yau 3-fold of Picard number one and degree $n$. Let $\NS(X)=\ZZ H$ and $m_H\in \ZZ_{>0}$ with $m_H H$ very ample.

Then $$\DT_{s,d}=\PT_{s,d}=0$$
when $d\geq N_H$ and
\[s<-(\frac{1}{2n}d^2+\frac{nm_H^3-4m_H}{2}d-\epsilon(m_{H} d, nm^2_{H})),\]
where $N_H$ is defined in \eqref{eq-sum-thm} and $\epsilon(-,-)$ is defined in \eqref{eq-epsilon-def}.
\end{theorem}

\begin{proof}
The vanishing of $\DT_{s,d}$ directly follows from Corollary \ref{cor-pic-rk-1}, since the corresponding moduli space is empty. By \cite[Corollary 5.1]{liu-ruan:cast-bound}, we also have the desired result for $\PT_{s,d}$.
\end{proof}

Now, by Theorem \ref{thm-vanishing} and a calculation of generating series \eqref{GV_PT_5}, we have the following result.

\begin{theorem}\label{thm-vanishing-GV}
Let $X$ be a projective Calabi--Yau 3-fold of Picard number one and degree $n$. Let $\NS(X)=\ZZ H$ and $m_H\in \ZZ_{>0}$ with $m_H H$ very ample.

Then $$\GV_{g,d}=0$$
when $d\geq \frac{n^3m_H^5-n^2m_H^3}{4}+nm_H(m_H^2N_H^2-n+1)$ and
\[g> \frac{1}{2n}d^2+\frac{nm_H^3-4m_H}{2}d+1-\epsilon(m_{H} d, nm^2_{H}),\]
where $N_H$ is defined in \eqref{eq-sum-thm} and $\epsilon(-,-)$ is defined in \eqref{eq-epsilon-def}.
\end{theorem}

\begin{proof}
We define a function $f(x)\colon \ZZ_{\geq 1}\to \QQ$ by
\[f(x)=\frac{1}{2n_H}x^2+\frac{n_H-4}{2}x+1-\epsilon(x,n_H)\]
when $x\geq m_HN_H$ and
\[f(x)=\frac{(x-1)(x-2)}{2}\]
otherwise. Note that $f(m_H d)=\frac{1}{2n}d^2+\frac{nm_H^3-4m_H}{2}d+1-\epsilon(m_{H} d, nm^2_{H})$ and
\[\frac{f(x)-1}{r}+1\geq f(\frac{x}{r})\]
for any integers $r\geq 1$ and $x\geq \frac{n^3m_H^5-n^2m_H^3}{4}+nm_H(m_H^2N_H^2-n+1)$. Then as in \cite[Proposition 5.3]{liu-ruan:cast-bound}, the desired vanishing of $\GV_{g,d}$ follows from a computation using \eqref{GV_PT_5}, Lemma \ref{lem-gv}, and the vanishing result for $\PT_{s,d}$ in Theorem \ref{thm-vanishing}.
\end{proof}

\begin{lemma}\label{lem-gv}
Given positive integers $N$ and $n$ with $N^2\geq n-1$, we define $f(x)\colon \ZZ_{\geq 1}\to \QQ$ by
\[f(x)=\frac{1}{2n}x^2+\frac{n-4}{2}x+1-\epsilon(x,n)\]
when $x\geq N$ and
\[f(x)=\frac{(x-1)(x-2)}{2}\]
otherwise. Then for any integer $x\geq \frac{n^3-n^2}{4}+ n(N^2-n+1)$ and any partition $\sum^s_{i=1} x_i=x$ with $x_i\in \ZZ_{\geq 1}$, we have
\begin{equation}\label{eq-lem-gv}
    f(x)-1\geq \sum^s_{i=1} (f(x_i)-1).
\end{equation}
\end{lemma}

\begin{proof}
When $x_i\geq N$ for any $1\leq i\leq s$, \eqref{eq-lem-gv} is equivalent to
\[\frac{1}{2n}x^2-\epsilon(x,n)\geq \sum^s_{i=1}\frac{1}{2n}x_i^2-\epsilon(x_i,n).\]
As $0\leq \epsilon(x,n)\leq \frac{n^2-n}{8}$ by \eqref{eq-epsilon-bound}, to prove \eqref{eq-lem-gv}, we only need to show
\[\frac{\sum_{i\neq j} x_i x_j}{n}\geq \frac{n^2-n}{8}.\]
Note that
\[\frac{\sum_{i\neq j} x_i x_j}{n}\geq \frac{Nx}{n},\]
then the assumption
\[x\geq \frac{n^3-n^2}{4}+ n(N^2-n+1)\geq \frac{n^3-n^2}{8N}\]
gives the result.

Now we assume that there exists $i$ such that $x_i\leq N$. Without loss of generality, we can assume that there is an integer $1\leq t\leq s$ such that $x_i\leq N$ for each $1\leq i\leq t$. Then \eqref{eq-lem-gv} is equivalent to
\begin{equation}
    \frac{1}{2n}x^2+\frac{n-1}{2}\sum^t_{i=1} x_i-\epsilon(x,n)\geq \sum^t_{i=1} \frac{1}{2}x^2_i+\sum^s_{i=t+1} (\frac{1}{2n}x^2_i-\epsilon(x_i,n)),
\end{equation}
which is also equivalent to
\begin{equation}\label{eq-lem-gv-1}
    \frac{1}{2n}x^2+\frac{n-1}{2}x-\epsilon(x,n)\geq \sum^t_{i=1} \frac{1}{2}x^2_i+\sum^s_{i=t+1} (\frac{1}{2n}x^2_i+\frac{n-1}{2}x_i-\epsilon(x_i,n)).
\end{equation}
As $0\leq \epsilon(x_i,n)$ and $x_i\leq N$, we have $\sum^t_{i=1} \frac{1}{2}x^2_i\leq \frac{t}{2}N^2$. Thus from \eqref{eq-lem-gv-1}, to prove \eqref{eq-lem-gv}, it suffices to prove
\begin{equation}
\frac{1}{2n}x^2+\frac{n-1}{2}x-\epsilon(x,n)\geq  \frac{t}{2}N^2+ \frac{1}{2n}(\sum^s_{i=t+1}x_i)^2+\sum^s_{i=t+1} \frac{n-1}{2}x_i.
\end{equation}
To this end, note that $\sum^t_{i=1} x_i\geq t$, hence $\sum^s_{i=t+1}x_i\leq x-t$, and we only need to prove
\begin{equation}\label{eq-lem-gv-3}
\frac{1}{2n}x^2+\frac{n-1}{2}x-\epsilon(x,n)\geq  \frac{t}{2}N^2+ \frac{1}{2n}(x-t)^2+\frac{n-1}{2}(x-t).
\end{equation}
Now we know that \eqref{eq-lem-gv-3} is equivalent to
\begin{equation}
    -\epsilon(x,n)\geq \frac{t}{2}N^2-\frac{t}{n}x+\frac{t^2}{2n}-\frac{(n-1)t}{2},
\end{equation}
which is also equivalent to 
\begin{equation}\label{eq-lem-gv-4}
    x\geq \frac{t}{2}+\frac{n\epsilon(x,n)}{t}+\frac{n}{2}(N^2-n+1).
\end{equation}
As $1\leq t\leq x$ and $\epsilon(x,n)\leq \frac{n^2-n}{8}$ by \eqref{eq-epsilon-bound}, the assumption $x\geq \frac{n^3-n^2}{4}+ n(N^2-n+1)$ implies \eqref{eq-lem-gv-4} as desired.
\end{proof}

\subsubsection{Surface-counting invariants of Calabi--Yau 4-folds}

Finally, we assume that $X$ is a projective Calabi--Yau 4-fold and fix a class 
\[v=\big(0,0,\gamma,\beta, n-\gamma.\td_2(X)\big)\in \mathrm{H}^*(X, \QQ).\]
For any $q\in \{-1,0,1\}$, let $\mathscr{P}_v^{(q)}(X)$ be the corresponding moduli space of $\PT_q$-stable pair on $X$ (cf.~\cite[Definition 2.1]{bae-kool-park:count-surface-I}).

As a direct consequence of Corollary \ref{cor-bound-cy4}, we get a vanishing theorem for any enumerative invariant defined over $\mathscr{P}_v^{(q)}(X)$:

\begin{corollary}\label{cor-vanish-cy4}
Let $X$ be projective Calabi--Yau 4-fold of Picard number one. Assume that $\NS(X)=\ZZ H$ with $H$ very ample and $n:=H^4$. Then there exists an integer $N_H$ defined in \eqref{eq-sum-thm}, such that for any class
\[v=\big(0,0,\gamma,\beta, n-\gamma.\td_2(X)\big)\in \mathrm{H}^*(X, \QQ),\]
if $d:=\gamma.H^2\geq N_H$ and 
\[\beta.H<-(\frac{1}{2n}d^2+\frac{n-5}{2}d-\epsilon(d, n)),\]
then
\[\mathscr{P}_v^{(q)}(X)=\varnothing\]
for any $q\in \{-1,0,1\}$.
\end{corollary}

\bibliographystyle{plain}
{\small{\bibliography{Cast}}}

\end{document}